\documentclass[11pt]{amsart}

\usepackage{amsmath}              
\usepackage{amsfonts}              
\usepackage{amsthm} 
\usepackage{algorithm}
\usepackage{algorithmic}
\usepackage[pdftex]{graphicx}
\usepackage{graphics}
\usepackage{mdwlist}
\usepackage{amssymb}
\usepackage{multicol}
\usepackage{pifont}
\usepackage{lscape}
\usepackage{float}

\newtheorem{thm}{Theorem}[section]
\newtheorem{lem}[thm]{Lemma}
\newtheorem{prop}[thm]{Proposition}

\newtheorem{defn}[thm]{Definition}
\newtheorem*{thm:mainthm}{Theorem \ref{thm:mainthm}}

\allowdisplaybreaks[3]  

\begin{document}

\title[Generic initial system of a 2-complete intersection]{The generic initial ideals of powers of a 2-complete intersection}
\author{Sarah Mayes}

\maketitle

\vspace*{-2em} 

\begin{abstract}
We compute the reverse lexicographic generic initial ideals of the powers of a $2$-complete intersection ideal $I$.  In particular, we give six algorithms to compute these generic initial ideals, the choice of which depends on the power and on the relative degrees of the minimal generators of $I$.
\end{abstract}

\section{Introduction}

Consider the collection of ideals $\{\text{gin}(I^n)\}_n$ obtained by taking the generic initial ideals of powers of a fixed ideal $I$ in a polynomial ring.  Our study of such families of monomial ideals was initially motivated by the desire to understand their asymptotic behaviour (see \cite{Mayes12}).   It soon became clear, however, that the individual ideals within such families are interesting in their own right.  In this paper we compute the generators of the ideals $\text{gin}(I^n)$ with respect to the reverse lexicographic order where $I$ is a 2-complete intersection and, in doing so, demonstrate relationships between such ideals.

Computing generic initial ideals is generally challenging because they are defined by an existence theorem rather than an explicit construction (see Galligo's Theorem, Theorem \ref{thm:galligo}).  As a result, there are few classes of ideals for which generic initial ideals have been explicitly computed (see \cite{Green98} for a survey, or \cite{Cimpoeas06}, \cite{ACP07}, \cite{CP08}, and \cite{CR10} for more recent results).

The 2-complete intersections are amongst the ideals whose reverse lexicographic generic initial ideals are completely understood.  In particular, if $I\subset K[x_1, \dots, x_m]$ is generated by a regular sequence of homogeneous polynomials of degrees $\alpha$ and $\beta$, with $\alpha \leq \beta$, then 
$$\text{gin}(I) = (x_1^{\alpha},x_1^{\alpha-1}x_2^{\lambda_{0}-2(\alpha-1)}, x_1^{\alpha-2}x_2^{\lambda_0-2(\alpha-2)}, \dots, x_1x_2^{\lambda_0-2}, x_2^{\lambda_0})$$ 
where $\lambda_0 = \beta+\alpha -1$ (see Section 4 of  \cite{Green98}).  The generic initial ideals for larger complete intersections, however, have proven difficult to compute.  For example, Cimpoea\c{s} \cite{Cimpoeas06} has exhibited the minimal generators for the generic initial ideals of strongly Lefschetz  3-complete intersections; the structure of such generic initial ideals is relatively difficult to describe and depends on the relative degrees of the generators of the complete intersection.

In this paper we explicitly compute the generators of the reverse lexicographic generic initial ideals of powers of 2-complete intersections.  In particular, we prove the following result.

\begin{thm:mainthm}
Fix positive integers $\alpha$, $\beta$, and $n$ such that $\beta \geq \alpha$ and $n \geq 2$.  If $I$ is a type $(\alpha, \beta)$ complete intersection in $K[x_1, \dots, x_m]$, where $K$ is a field of characteristic 0, then the reverse lexicographic generic initial ideal of $I^n$ is
$$\text{gin}(I^n) = (x_1^k, x_1^{k-1}x_2^{\lambda_{k-1}}, \dots, x_1x_2^{\lambda_1}, x_2^{\lambda_0})$$
where $k = n\alpha$ and $\{\lambda_i\}$ is the sequence of natural numbers arising from:
\begin{itemize}
\item Algorithm \ref{alg:far} if $\beta \geq 2\alpha-1$;
\item Algorithm \ref{alg:mid} if $2\alpha -1 > \beta \geq \frac{3}{2}\alpha$;
\item Algorithm \ref{alg:closedivides} if $\frac{3}{2}\alpha > \beta >\alpha$, $(\beta-\alpha) | \alpha$, and $n \geq \frac{\alpha}{\beta-\alpha} +1$;
\item Algorithm \ref{alg:closedoesnotdivide} if $\frac{3}{2}\alpha > \beta >\alpha$, $(\beta-\alpha) \nmid \alpha$, and $n \geq \lceil \frac{\alpha}{\beta-\alpha} \rceil+1$;
\item Algorithm \ref{alg:closesmalln} if $\frac{3}{2}\alpha > \beta >\alpha$ and $2 \leq n < \lceil \frac{\alpha}{\beta-\alpha} \rceil +1$; and
\item Algorithm \ref{alg:equal} if $\alpha = \beta$.
\end{itemize} 
\end{thm:mainthm}

The algorithms referred to in this theorem are stated in Section \ref{sec:proposedinvariants}. Although the particular choice of an algorithm in the theorem depends on $n$ and on the relative sizes of $\alpha$ and $\beta$, all of the algorithms share common features.  For example, they each compute the invariants $\lambda_i$ one-by-one, starting with $\lambda_0 =n\beta+\alpha-1$ and using the gaps $g_i=\lambda_{i-1}-\lambda_i$  to compute each successive invariant.  The patterns amongst the invariants of the ideals $\text{gin}(I^n)$ are best seen  by looking at the associated gap sequences of $\{g_i\}$, which consist entirely of the numbers 1, 2, and $\beta-2\alpha+2$.

This theorem adds powers of 2-complete intersections to the classes of ideals whose generic initial ideals can be explicitly computed.  The complexity of this result even in this small case, however, gives further evidence that finding generators of the generic initial ideals of powers of larger complete intersections may be optimistic and provides motivation to instead study the asymptotic behaviour of generic initial systems $\{\text{gin}(I^n)\}_n$.  One consequence of Theorem \ref{thm:mainthm} is a different proof for the case of a 2-complete intersection of the main result of \cite{Mayes12} on the asymptotic behaviour of the generic initial system.
\section{Preliminaries}
\label{sec:prelim}

In this section we will introduce some notation, definitions, and preliminary results related to generic initial ideals. Throughout, $R=K[x_1, \dots, x_m]$ is a polynomial ring over a field $K$ of characteristic 0 with the standard grading and some fixed term order $>$ with $x_1 > x_2 > \cdots >x_m$.  

\subsection{Generic Initial Ideals}

An element $g = (g_{ij}) \in \text{GL}_m(K)$ acts on $R$ and sends any homogeneous element $f(x_1, \dots, x_m)$ to the homogeneous element 
$$f(g(x_1), \dots, g(x_m))$$ 
where $g(x_i) = \sum_{j=1}^m g_{ij}x_j$.  If $g(I)=I$ for every upper triangular matrix $g$ then we say that $I$ is \textit{Borel-fixed}.  Borel-fixed ideals are \textit{strongly stable} when $K$ is of characteristic 0; that is, for every monomial $f$ in the ideal such that $x_i$ divides $f$, the monomials $\frac{x_jf}{x_i}$ for all $j<i$ are also in the ideal.  This property makes such ideals particularly nice to work with.

To any homogeneous ideal $I$ of $R$ we can associate a Borel-fixed monomial ideal $\text{gin}_{>}(I)$ which can be thought of as a coordinate-independent version of the initial ideal.\footnote{For a polynomial $f= \sum a_i m_i$, $\text{in}_>(f)$ is the largest $m_i$ with respect to $>$ such that $a_i$ is nonzero.  Further, for a polynomial ideal $I$, $\text{In}_>(I) = \{ \text{in}(f) : f \in I\}$.}  Its existence is guaranteed by the following result known as Galligo's theorem (also see \cite[Theorem 1.27]{Green98}).

\begin{thm}[{\cite{Galligo74} and \cite{BS87b}}]
\label{thm:galligo}
For any multiplicative monomial order $>$ on $R$ and any homogeneous ideal $I\subset R$, there exists a Zariski open subset $U \subset \text{GL}_m$ such that $\text{In}_{>}(g(I))$ is constant and Borel-fixed for all $g \in U$.  
\end{thm}

\begin{defn}
The \textbf{generic initial ideal of $I$}, denoted $\text{gin}_{>}(I)$, is defined to be $\text{In}_{>}(g(I))$ where $g \in U$ is as in Galligo's theorem.
\end{defn}

The \textit{reverse lexicographic order} $>$ is a total ordering on the monomials of $R$ defined by: 
\begin{enumerate}
\item if $|I| =|J|$ then $x^I > x^J$ if there is a $k$ such that $i_m = j_m$ for all $m>k$ and $i_k < j_k$; and
\item if $|I| > |J|$ then $x^I >x^J$.
\end{enumerate}
For example, $x_1^2x_3<x_1x_2^2$.   From this point on, $\text{gin}(I) = \text{gin}_{>}(I)$ will denote the generic initial ideal with respect to the reverse lexicographic order.

\subsection{The Hilbert Function and Notation}

Recall that the Hilbert function $H_I(t)$ of a homogeneous ideal $I$ is defined by $H_I(t) = \text{dim}_{K} (I_t)$ where $I_t$ denotes the $t^{\text{th}}$ graded piece of $I$.  The following theorem records two of the properties shared by $\text{gin}(I)$ and $I$.  The first statement is a consequence of the fact that Hilbert functions are invariant under making changes of coordinates and taking initial ideals.  The second statement is a result of Bayer and Stillman \cite{BS87}; for a simple proof see Corollary 2.8 of \cite{AhnMigliore07}.

\begin{thm}
\label{thm:commonproperties}
For any homogeneous ideal $I$ in $R$:
\begin{enumerate}
\item the Hilbert functions of $I$ and $\text{gin}(I)$ are equal; and
\item under the reverse lexicographic order, \text{depth}(R/I) = \text{depth}(R/\text{gin}(I)).
\end{enumerate}
\end{thm}

Throughout this paper, ${s \choose t}=0$ whenever $s\leq 0$ or $t>s$ so that ${s \choose t}$ is always nonnegative.  Under this assumption, the summation and recursive formulas for binomial coefficients hold:

\begin{equation}\label{eq:Summation} \sum_{z=0}^L {z \choose p} = {L+1 \choose p+1}
\end{equation}

\begin{equation}\label{eq:Recursive} {z\choose p} = {z-1 \choose p-1} + {z-1 \choose p}
\end{equation}

As a consequence of Equation \ref{eq:Summation} we have the following identities when $L_1, L_2  > 0$:

\begin{equation}\label{eq:SummationCor1} \sum_{j=L_1}^{L_2} {z+j \choose p} = {z+L_2+1 \choose p+1} - {z+L_1 \choose p+1}
\end{equation}

\begin{equation}\label{eq:SummationCor2} \sum_{j=L_1}^{L_2} {z-j \choose p} = {z-L_1+1 \choose p+1} - {z-L_2 \choose p+1}
\end{equation}

These are the only binomial coefficient identities that will be used in our later calculations.

Finally, since most of our work will only involve the first two variables $x_1$ and $x_2$ of $K[x_1, \dots, x_m]$, we will set $x_1 = x$ and $x_2 = y$. 
\section{Structure of ideals in the generic initial system}

A homogeneous ideal $I = (f_{\alpha}, f_{\beta})$ is a \textit{complete intersection of type} $(\alpha, \beta)$ if $f_{\alpha}, f_{\beta}$ is a regular sequence on $R$, $\text{deg}(f_{\alpha}) = \alpha$, and $\text{deg}(f_{\beta}) = \beta$.  Since $f_{\alpha}$ and $f_{\beta}$ are homogeneous, $f_{\beta}, f_{\alpha}$ is also a regular sequence; therefore, we may assume that $\alpha \leq \beta$. Throughout this section we assume that $I$ is such a complete intersection.

\subsection{Structure of $\text{gin}(I^n)$}\label{sec:ginstructure}
The goal of this subsection is to describe the general structure of the reverse lexicographic generic initial ideals $\text{gin}(I^n)$ for a complete intersection $I$ of type $(\alpha, \beta)$.  In particular, we will prove the following theorem.

\begin{thm}  
\label{thm:ginstructure}
Let $I$ be a complete intersection of type $(\alpha,\beta)$ in $R = K[x_1, \dots, x_m]$ generated by the homogeneous polynomials $f_{\alpha}$ and $f_{\beta}$, and suppose that $A_n$ is the set of minimal monomial generators of $\text{gin}(I^n)$.  Then, setting $x=x_1$ and $y=x_2$, 
$$A_n = \{x^k, x^{k-1}y^{\lambda_{k-1}}, x^{k-2}y^{\lambda_{k-2}}, \dots, xy^{\lambda_1}, y^{\lambda_0}\}$$
where 
\begin{itemize}
\item[ (i) ] $\lambda_0  \gneq \lambda_1 \gneq \cdots \gneq \lambda_{k-2} \gneq \lambda_{k-1}$;
\item[ (ii) ]  $k=n\alpha$;
\item[ (iii) ]  $\lambda_0 = n\beta+\alpha-1$; and
\item[ (iv) ]  $\lambda_{k-1} = \beta-\alpha+1.$
\end{itemize}
\end{thm}

We will refer to the $\lambda_i$ as the \textit{invariants} of $\text{gin}(I^n)$.  This theorem will be proven in several parts.  First, no matter how many variables the ambient ring $R$ has, the minimal generators of these generic initial ideals will only involve the variables $x_1$ and $x_2$.

\begin{lem}\label{lem:varsingenset}  
Let $I$ be a type $(\alpha, \beta)$ complete intersection in $R$ and let $A_n$ denote the set of minimal monomial generators of $\text{gin}(I^n)$.  Then the elements of $A_n$ are contained in $K[x_1, x_2]$.  Furthermore, $A_n$ contains a power of $x_2$, say $x_2^{\lambda_0}$, and no element of $A_n$ is of degree greater than $\lambda_0$. 
\end{lem}

This lemma is a consequence of the following result of Herzog and Srinivasan (see Lemma 3.1 of \cite{HS98}) which relates the depth and dimension of a Borel-fixed monomial ideal to the variables appearing in its minimal generating set.  

\begin{prop}
\label{prop:DMrelations}
Let $J$ be a Borel-fixed monomial ideal in $R$ and define 
$$D(J) := \text{max} \{t | x_t^j \in J \text{ for some positive integer } j \}$$
and 
$$M(J) := \text{max} \{t | x_t \text{ appears in some minimal generator of } J \}.$$
Then 
\begin{enumerate}
\item $\text{dim}(R/J) = m-D(J)$; and
\item $\text{depth}(R/J) = m-M(J)$.
\end{enumerate}
\end{prop}

Note that when $I$ is a complete intersection of type $(\alpha, \beta)$ in $R$,
$$\text{dim}(R/I^n) = \text{depth}(R/I^n) = m-2$$
for all $n \geq 1$.
It then follows by Theorem \ref{thm:commonproperties} that the depth and dimension of $R/\text{gin}(I^n)$ are equal to $m-2$ as well.

\begin{proof}[Proof of Lemma \ref{lem:varsingenset}]
By Proposition \ref{prop:DMrelations},
$$D(\text{gin}(I^n)) = m-\text{dim}(R/\text{gin}(I^n)) = 2 = m- \text{depth}(R/\text{gin}(I^n)) = M(\text{gin}(I^n)).$$
This means that the minimal monomial generating set $A_n$ of $\text{gin}(I^n)$ is contained in $S=K[x_1, x_2]$ and that $A_n$ contains a power of $x_2$, say $x_2^{\lambda_0}$. 
The fact that $\text{gin}(I^n)$ is strongly stable means that we can replace any number of $x_2$ variables in $x_2^{\lambda_0}$ with $x_1$ and still get an element of $\text{gin}(I^n)$.  Therefore, any monomial $x^J \in S = K[x_1, x_2]$ of degree $\lambda_0$ is also contained in $\text{gin}(I^n)$.  Now it is clear that the set of minimal monomial generators, $A_n \subset S$, cannot contain any element of degree greater than $\lambda_0$.
\end{proof}

\begin{proof}[Proof of Theorem \ref{thm:ginstructure} (i)]
By Lemma \ref{lem:varsingenset}, $A_n \subset K[x,y]$ and $y^{\lambda_0} \in A_n$.  Let $m_k$ be a monomial of least degree $k$ in $\text{gin}(I^n)$.  Since $\text{gin}(I^n)$ is strongly stable, every variable appearing in $m_k$ can be replaced by $x$ and still stay inside of $\text{gin}(I^n)$; thus, $x^k \in \text{gin}(I^n)$ and, as a least degree element, $x^k$ is also in $A_n$.  Define $\lambda_i$ by $\lambda_i = \min\{t | x^iy^t \in \text{gin}(I^n)\}$ so that
$$A_n \subset \{x^k, x^{k-1}y^{\lambda_{k-1}}, x^{k-2}y^{\lambda_{k-2}}, \dots, xy^{\lambda_1}, y^{\lambda_0}\} \subset \text{gin}(I^n).$$

Since $x^iy^{\lambda_i}$ is in the strongly stable ideal $\text{gin}(I^n)$,
$$\frac{x(x^iy^{\lambda_i})}{y} = x^{i+1}y^{\lambda_i-1} \in \text{gin}(I^n)$$
for all $i = 1, \dots, k-2$.  This condition holds if and only if $\lambda_i-1 \geq \lambda_{i+1}$, or $\lambda_i \gneq \lambda_{i+1}$.  Therefore, the $\lambda_i$s are strictly decreasing and $\lambda_{k-1} \geq 1$.  Thus,
$$A_n= \{x^k, x^{k-1}y^{\lambda_{k-1}}, x^{k-2}y^{\lambda_{k-2}}, \dots, xy^{\lambda_1}, y^{\lambda_0}\}.$$
\end{proof}

\begin{proof}[Proof of Theorem~\ref{thm:ginstructure}(ii)]
Note that, since $\alpha \leq \beta$, the homogeneous polynomial $f_{\alpha}^n$ is an element of $I^n$ of the smallest degree. Under a general change of coordinates $g$, the smallest degree element of $g(I^n)$ is also of degree $n\alpha$ and its initial term is of degree $n\alpha$.  Thus, the smallest degree element of  $\text{in}(g(I^n)) = \text{gin}(I^n)$ has degree $n\alpha$ and, since $\text{gin}(I^n)$ is strongly stable, this is equal to the power of $x$ in $A_n$.
\end{proof}

To determine the values of $\lambda_0$ and $\lambda_{k-1}$ we will compare the \textit{Betti numbers} of $I^n$ and $\text{gin}(I^n)$ using `The Cancellation Principle'. Let
$$0 \rightarrow F_m \rightarrow \cdots \rightarrow F_1 \rightarrow F_0 \rightarrow J \rightarrow 0$$
be the unique minimal free graded resolution of a homogeneous ideal $J$.  The graded Betti numbers of $J$, $\beta_{i,j}(J)$, are defined by $F_i = \bigoplus_j R(-j)^{\beta_{i,j}(J)}$.
A \textit{consecutive cancellation} takes a sequence $\{\beta_{i,j}\}$ to a new sequence by replacing $\beta_{i,j}$ by $\beta_{i,j}-1$ and $\beta_{i+1,j}$ by $\beta_{i+1,j}-1$.  The `Cancellation Principle' says that the graded Betti numbers $\beta_{i,j}(I^n)$ of $I^n$ can be obtained from the graded Betti numbers $\beta_{i,j}(\text{gin}(I^n))$ of $\text{gin}(I^n)$  by making a series of consecutive cancellations (see Corollary 1.21 of \cite{Green98}).

In order to apply the Cancellation Principle to find $\lambda_{k-1}$ and $\lambda_0$, we need to know the Betti numbers of $I^n$ and an ideal having the same form as $\text{gin}(I^n)$; this information is recorded in the following two propositions.

 \begin{prop}[\cite{Guardo05}]\label{prop:Inres}
 Suppose that $I$ is a complete intersection of type $(\alpha,\beta)$.  Then the minimal free resolution of $I^n$ is of the form
\begin{equation*}\label{Inres}
0\rightarrow \mathcal{H}_1 \rightarrow\mathcal{H}_0 \rightarrow I^n \rightarrow 0
\end{equation*}
where 
$$\mathcal{H}_1 = \bigoplus_{p=1}^n R(-\alpha p - \beta (n+1-p))$$
and
$$\mathcal{H}_0 = \bigoplus_{p=0}^n R(-\alpha p-\beta (n-p))=R(-\alpha n) \oplus \bigoplus_{p=0}^{n-1} R(-\alpha p -\beta (n-p)).$$
\end{prop}

\begin{prop}[cf \cite{EK90}]\label{prop:Jres}
The minimal free resolution of \newline $J= (x^k, x^{k-1}y^{\lambda_{k-1}}, \dots, xy^{\lambda_1}, y^{\lambda_0})$ where $\lambda_0 > \lambda_1 > \cdots \lambda_{k-1}$ is of the form 
\begin{equation*}
0 \rightarrow \mathcal{G}_1 \rightarrow \mathcal{G}_0 \rightarrow J \rightarrow 0
\end{equation*}
where
$$\mathcal{G}_1= \bigoplus_{i=0}^{k-1} R(-\lambda_i-i-1)$$
and
$$\mathcal{G}_0 = \big(\bigoplus_{i=0}^{k-1} R(-\lambda_i-i)\big) \oplus R(-k). $$
\end{prop}

\begin{proof}[Proof of Theorem~\ref{thm:ginstructure}(iii)]
Since the invariants $\lambda_i$ are strictly decreasing, $\lambda_0 +1 > \lambda_i+i \geq k$ for $i = 0, \dots, k-1$. Thus, if $\{\beta_{i,j}\}$ is the set of graded Betti numbers of $\text{gin}(I^n)$, $\beta_{1, \lambda_0+1} \geq 1$ and $\beta_{0, \lambda_0+1} = 0$ by Proposition \ref{prop:Jres}.  Therefore, no consecutive cancellation can replace $\beta_{1, \lambda_0+1}$ and after any series of consecutive cancellations 
$$\text{max}\{t | \beta_{1,t} \geq 1\} = \lambda_0+1.$$
By Proposition \ref{prop:Inres}, $\alpha + n\beta$ is the largest shift in $\mathcal{H}_1$.  Thus, by the Cancellation Principle, $\lambda_0+1 = \alpha+n\beta$, or
\[\lambda_0=\alpha+n\beta-1. \qedhere \]
\end{proof}

\begin{proof}[Proof of Theorem~\ref{thm:ginstructure}(iv)]
Since the invariants $\lambda_i$ are strictly decreasing and $\lambda_{k-1} \geq 1$, $k \leq \lambda_{k-1} +(k-1) <  \lambda_i+i +1$ for all $i = 0, \dots, k-1$. Thus, if $\{\beta_{i,j}\}$ is the set of graded Betti numbers of $\text{gin}(I^n)$, $\beta_{0,k} \geq 1$, $\beta_{0, \lambda_{k-1}+k-1} \geq 1$, $\beta_{1,k}=0$, and $\beta_{1, \lambda_{k-1}+k-1} = 0$ by Proposition \ref{prop:Jres}. Therefore, no consecutive cancellation can replace $\beta_{0,k}$ or $\beta_{0, \lambda_{k-1}+k-1}$  and, for every $t$ such that $t<k$ or $k<t<\lambda_{k-1}+k-1$, $\beta_{0, t} = 0$ (note that it is possible to have $k = \lambda_{k-1}+k-1$).  

By Proposition \ref{prop:Inres}, the two smallest shifts in $\mathcal{H}_0$ are $n\alpha$ and $\alpha(n-1)+\beta$.  Thus, by the Cancellation Principle, $k=n\alpha$ (as we have seen in the proof of part (ii)) and $\lambda_{k-1} +k-1= \lambda_{n\alpha-1} +n\alpha-1 = \alpha(n-1)+\beta$, or
\[ \lambda_{n\alpha-1} = \beta-\alpha+1. \qedhere\]
\end{proof}

Note that we can write $\lambda_0$ and $\lambda_{k-1}$ in terms of $l := \beta-\alpha$ and $\alpha$ as follows:
$$\lambda_0 = n(\alpha+l)+\alpha-1 = (n+1)\alpha+nl-1$$
$$\lambda_{k-1} = \lambda_{n\alpha-1} = \beta-\alpha+1 = l+1.$$


\subsection{The Hilbert function of $\text{gin}(I^n)$}
\label{sec:hilbmethod}

The following result tells us that the invariants of $\text{gin}(I^n)$ are completely determined by $H_{\text{gin}(I^n)}(t)$; this observation will be the key to computing these invariants.

\begin{lem}
\label{lem:determineinvariants}
Suppose that we have an ideal $J$ of the form 
$$J = (x^{k}, x^{k-1}y^{\mu_{k-1}}, \dots, xy^{\mu_1}, y^{\mu_0})$$
where the $\mu_i$s are strictly decreasing.  If $H_J(t) = H_{I^n}(t)$ for a type $(\alpha,\beta)$ complete intersection ideal $I$ then 
$$\text{gin}(I^n) = J.$$
\end{lem}

This lemma is an immediate consequence of the following well-known result.  Although it is used in the literature (for example, it has the same content as Lemma 4.2 of \cite{Green98}), we record a complete proof here.

\begin{lem}
\label{lem:uniquehilbert}
An ideal of the form
$$J=(x^k, x^{k-1}y^{\lambda_{k-1}}, \dots, xy^{\lambda_1}, y^{\lambda_0})$$
where $\lambda_0 > \lambda_1 > \cdots > \lambda_{k-1}$ is uniquely determined by its Hilbert function.  
\end{lem}

\begin{proof}
The key observation here is that 
$$\text{deg}(x^iy^{\lambda_i}) = i +\lambda_i = (i-1)+(\lambda_i+1) \leq (i-1)+\lambda_{i-1} = \text{deg}(x^{i-1}y^{\lambda_{i-1}})$$ 
since $\lambda_{i} < \lambda_{i-1}$. 
Suppose that $H_J(t)$ is the Hilbert function of an ideal $J$ as in the statement of the lemma.  

First note that $x^k$ is the smallest degree element of $J$ so that $k = \min\{t | H_J(t) \neq 0\}$.  

Consider the ideal $L_k = (x^k) \subset J$ and its Hilbert function $H_{L_k}(t)$.  Set 
$$S_k = \min\{t | H_J(t) \neq H_{L_k}(t)\}$$
so that the smallest degree monomial that is in $J$ but not in $L_k$ is of degree $S_k$.  Since $\text{deg}(x^iy^{\lambda_i}) \leq \text{deg}(x^{i-1}y^{\lambda_{i-1}})$ for all $i$, we must have $\text{deg}(x^{k-1}y^{\lambda_{k-1}}) = S_k$ and $\lambda_{k-1} = S_k-(k-1)$.

The same argument works in general by induction.  Suppose that we have determined the values of $k$ and $\lambda_{k-1}, \dots, \lambda_T$ and let $L_T = (x^k,x^{k-1}y^{\lambda_{k-1}}, \dots, x^Ty^{\lambda_T}) \subset J$.  Set 
 $$S_T = \min\{t | H_J(t) \neq H_{L_T}(t)\}.$$
By the same argument as above, $\text{deg}(x^{T-1}y^{\lambda_{T-1}}) = S_T$ and $\lambda_{T-1} = S_T-(T-1)$.
\end{proof}

\begin{proof}[Proof of Lemma~\ref{lem:determineinvariants}]
By Theorem \ref{thm:commonproperties},
$$H_{\text{gin}(I^n)}(t) = H_{I^n}(t) = H_J(t).$$
Since, $J$ and $\text{gin}(I^n)$ are both of the form considered in Lemma \ref{lem:uniquehilbert}, they are uniquely determined by their Hilbert functions and $J = \text{gin}(I^n)$.
\end{proof}

To prove that the numbers $\{\lambda_i\}$ produced by the algorithms presented in Section \ref{sec:proposedinvariants} are indeed the invariants of $\text{gin}(I^n)$, we will compute the Hilbert function of the ideal 
$$J= (x^k, x^{k-1}y^{\lambda_{k-1}}, \dots, xy^{\lambda_1}, y^{\lambda_0}).$$
By Lemma \ref{lem:determineinvariants} it is then sufficient to show that $H_J(t)$ is equal to $H_{I^n}(t)$.  We will now record expressions for the Hilbert functions of $I^n$ and $J$ that will be used to carry out this procedure.


\begin{prop}\label{prop:HilbofIn}
If $I$ is the ideal of a type $(\alpha, \beta)$ complete intersection in $K[x_1, \dots, x_m]$ then 
\begin{eqnarray*}
H_{I^n}(t)&=&  \sum_{j=1}^n \Bigg [ {t-\alpha(n-j)-\beta j +(m\!\!-\!\!1) \choose (m\!\!-\!\!1)} - {t-\alpha j -\beta (n+1 -j)+(m\!\!-\!\!1)\choose (m\!\!-\!\!1)} \Bigg]  \\
{} {} {} {} {}  & & + {t-n\alpha+(m\!\!-\!\!1) \choose (m\!\!-\!\!1)}.
\end{eqnarray*}
Setting $l:=\beta-\alpha$,  
\begin{eqnarray*}
H_{I^n}(t)&=& \sum_{j=1}^n \Bigg[ {t-\alpha n -jl +(m\!\!-\!\!1) \choose (m\!\!-\!\!1)} - {t-\alpha(n+1) - lj + (m\!\!-\!\!1) \choose (m\!\!-\!\!1)}\Bigg]\\
{} {} {} {} {}  & & + {t-n\alpha+(m\!\!-\!\!1) \choose (m\!\!-\!\!1)}.
\end{eqnarray*}
\end{prop}

\begin{proof}
By Proposition \ref{prop:Inres}, 
\begin{eqnarray*} \label{hilbIn}
H_{I^n}(t) &=&  \text{dim}_K [R(-\alpha n)]_t + \sum_{p=0}^{n-1} \text{dim}_K [R(-\alpha p -\beta (n-p))]_t - \sum_{p=1}^n \text{dim}_K [R(-\alpha p-\beta(n+1-p))]_t\\
&=&  {t-n\alpha+(m\!\!-\!\!1) \choose (m\!\!-\!\!1)} + \sum_{p=0}^{n-1} {t-\alpha p-\beta (n-p)+(m\!\!-\!\!1) \choose (m\!\!-\!\!1)} \\
{} {} {} {} {}   & &-\sum_{p=1}^n {t-\alpha p-\beta(n+1-p)+(m\!\!-\!\!1) \choose(m\!\!-\!\!1)}\\
&=&  \sum_{j=1}^n \Bigg [ {t-\alpha(n-j)-\beta j +(m\!\!-\!\!1) \choose (m\!\!-\!\!1)} - {t-\alpha j -\beta (n+1 -j)+(m\!\!-\!\!1)\choose (m\!\!-\!\!1)} \Bigg]  \\
{} {} {} {} {}  & & + {t-n\alpha+(m\!\!-\!\!1) \choose (m\!\!-\!\!1)}
\end{eqnarray*}

If $l=\beta-\alpha$, the \textit{sum} in the above expression is

\begin{eqnarray*}
& & \sum_{j=1}^n \Bigg[ {t-\alpha n - j(\beta-\alpha)+(m\!\!-\!\!1) \choose (m\!\!-\!\!1)} - {t-\beta(n+1) - j(\alpha-\beta)+(m\!\!-\!\!1) \choose (m\!\!-\!\!1)} \Bigg] \\
&=& \sum_{j=1}^n \Bigg[ {t-\alpha n - j l +(m\!\!-\!\!1) \choose (m\!\!-\!\!1)} - {t-(\alpha+l)(n+1) + j l + (m\!\!-\!\!1)\choose (m\!\!-\!\!1)}\Bigg] \\
&=& \sum_{j=1}^n \Bigg[ {t-\alpha n -jl +(m\!\!-\!\!1) \choose (m\!\!-\!\!1)} - {t-\alpha(n+1) - l(n+1-j) + (m\!\!-\!\!1)\choose (m\!\!-\!\!1)}\Bigg] \\
&=& \sum_{j=1}^n \Bigg[ {t-\alpha n -jl +(m\!\!-\!\!1) \choose (m\!\!-\!\!1)} - {t-\alpha(n+1) - lj + (m\!\!-\!\!1) \choose (m\!\!-\!\!1)}\Bigg].\\
\end{eqnarray*}
Note that the last equality follows by changing the indexing.
\end{proof}

\begin{prop}\label{prop:HilbofJ}
Suppose that we have an ideal $J$ of the form
$$J= (x^k, x^{k-1}y^{\lambda_{k-1}}, \dots, xy^{\lambda_1}, y^{\lambda_0})$$
where $\lambda_0 \geq \lambda_1 \geq \cdots \geq \lambda_{k-1}$.  Then
\begin{eqnarray*} 
H_J(t) &=& \sum_{i=0}^{k-1} {t-\lambda_i-i+(m\!\!-\!\!2) \choose (m\!\!-\!\!2)} + {t-k+(m\!\!-\!\!1) \choose (m\!\!-\!\!1)}.\\
\end{eqnarray*}
\end{prop}

\begin{proof}
From Proposition \ref{prop:Jres},
\begin{eqnarray*} 
H_J(t) &=& \text{dim}_K[R(-k)] + \sum_{i=0}^{k-1} \text{dim}_K[R(-\lambda_i-i)] - \sum_{i=0}^{k-1} \text{dim}_K[R(-\lambda_i-i-1)]\\
&=& \sum_{i=0}^{k-1} {t-\lambda_i-i+(m\!\!-\!\!1) \choose (m\!\!-\!\!1)} - \sum_{i=0}^{k-1} {t-\lambda_i-i+(m\!\!-\!\!2) \choose (m\!\!-\!\!1)} + {t-k+(m\!\!-\!\!1) \choose (m\!\!-\!\!1)}\\
&=& \sum_{i=0}^{k-1} {t-\lambda_i-i+(m\!\!-\!\!2) \choose (m\!\!-\!\!2)} + {t-k+(m\!\!-\!\!1) \choose (m\!\!-\!\!1)}.\\
\end{eqnarray*}
The last equality follows from the recursive formula for binomial coefficients (see Equation \ref{eq:Recursive}). 
\end{proof}

\section{Main theorem and the proposed invariants}
\label{sec:proposedinvariants}

In the previous section we determined the general structure of $\text{gin}(I^n)$ where $I$ is a 2-complete intersection of type $(\alpha, \beta)$ and showed that it was defined by a strictly decreasing sequence of invariants $\{ \lambda_i \}$.  We also found expressions for $\lambda_0$ and $\lambda_{k-1}$ in terms of $n$, $\alpha$, and $\beta$ (see Theorem \ref{thm:ginstructure}).  In this section we propose algorithms to determine the remaining invariants, and thus the minimal generators, of $\text{gin}(I^n)$.  Throughout $l:=\beta-\alpha$.

\begin{thm}
\label{thm:mainthm}
Fix positive integers $\alpha$, $\beta$, and $n$ such that $\beta \geq \alpha$ and $n \geq 2$.  Compute the sequence of invariants $\{\lambda_i\}$ using:
\begin{itemize}
\item Algorithm \ref{alg:far} if $\beta \geq 2\alpha-1$;
\item Algorithm \ref{alg:mid} if $2\alpha -1 > \beta \geq \frac{3}{2}\alpha$;
\item Algorithm \ref{alg:closedivides} if $\frac{3}{2}\alpha > \beta >\alpha$, $(\beta-\alpha) | \alpha$, and $n \geq \frac{\alpha}{\beta-\alpha} +1$;
\item Algorithm \ref{alg:closedoesnotdivide} if $\frac{3}{2}\alpha > \beta >\alpha$, $(\beta-\alpha) \nmid \alpha$, and $n \geq \lceil \frac{\alpha}{\beta-\alpha} \rceil+1$;
\item Algorithm \ref{alg:closesmalln} if $\frac{3}{2}\alpha > \beta >\alpha$ and $2 \leq n < \lceil \frac{\alpha}{\beta-\alpha} \rceil +1$; and
\item Algorithm \ref{alg:equal} if $\alpha = \beta$.
\end{itemize} 
If $I$ is a type $(\alpha,\beta)$ complete intersection in $R$ then the reverse lexicographic generic initial ideal of $I^n$ is
$$\text{gin}(I^n) = (x^k, x^{k-1}y^{\lambda_{k-1}}, \dots, xy^{\lambda_1}, y^{\lambda_0})$$
where $k = n\alpha$.
\end{thm}

Each of the algorithms, and thus the invariants $\lambda_i$ that they produce and the resulting gaps $g_i:=\lambda_{i-1}-\lambda_i$,  can be divided into three consecutive phases which we refer to as the Build, the Pattern, and the Reverse Build.  As the names of the phases suggest, the gap sequences $\{g_i\}$ arising from the Reverse Build and the Build are almost mirror images of each other while the gap sequences arising from the Pattern consists of a number of repeats of the same sub-sequence called a Pattern Block.   


\subsection{Algorithms Producing the Proposed Invariants}
\label{sec:algorithms}

Given three positive integers $n$, $\alpha$, and $\beta$ where $n \geq 2$ and $\beta \geq \alpha$, the following algorithms produce a sequence of positive integers $\lambda_0, \dots, \lambda_{k-1}$ which Theorem \ref{thm:mainthm} claims are the invariants of $\text{gin}(I^n)$  for a type $(\alpha, \beta)$ complete intersection $I$.

For examples and illustrations of the outputs of these algorithms, see Appendix \ref{ap:examples}.  

\begin{algorithm}[H]
\caption{Determine $\{ \lambda_i \}$ for $\beta \geq 2\alpha-1$, $n \geq 1$}
\label{alg:far}
\begin{algorithmic}
\STATE $i=1$
\STATE $\lambda_0 = n\beta+\alpha-1$
\STATE $h=1$
\WHILE{$h \leq n-1$}
\STATE \textbf{BlockFar($i, \lambda_{i-1}, \alpha, \beta$)}
\STATE $h=h+1$
\ENDWHILE
\STATE \textbf{PartialBlockFar($i, \lambda_{i-1}, \alpha$)}
\end{algorithmic}
\end{algorithm}

\textbf{Sub-routines for Algorithm \ref{alg:far}}
\begin{multicols}{2}
\begin{algorithmic}
\STATE \textbf{BlockFar($i, \lambda_{i-1}, \alpha, \beta$)}
\STATE $t=1$
\WHILE{$t \leq \alpha-1$}
\STATE $\lambda_i=\lambda_{i-1}-2$
\STATE $t=t+1$
\STATE$i=i+1$
\ENDWHILE
\STATE $\lambda_i = \lambda_{i-1}-(\beta-2\alpha+2)$
\STATE $i=i+1$
\STATE \textbf{RETURN}
\end{algorithmic}

\columnbreak

\begin{algorithmic}
\STATE \textbf{PartialBlockFar($i, \lambda_{i-1}, \alpha$)}
\STATE $h=1$
\WHILE{$h\leq \alpha-1$}
\STATE $\lambda_i=\lambda_{i-1}-2$
\STATE $i=i+1$
\STATE $h=h+1$
\ENDWHILE
\STATE \textbf{RETURN}
\end{algorithmic}

\end{multicols}

The subroutines not called by Algorithm \ref{alg:far} appear after the statements of the remaining algorithms.


\begin{algorithm}
\caption{Determine $\{ \lambda_i \}$ for $2\alpha-1>\beta \geq \frac{3}{2}\alpha$, $n \geq 2$}
\label{alg:mid}
\begin{algorithmic}
\STATE $l=\beta-\alpha$
\STATE $r=2\alpha-\beta$
\STATE $\lambda_0 = n\beta+\alpha-1$
\STATE $i=1$
\STATE \textbf{Build($0, i, \lambda_{i-1}$)}
\IF{$n \geq 3$}
\STATE $h=1$
\WHILE{$h \leq n-2$}
\STATE \textbf{BlockMid($i, \lambda_{i-1}, r, \alpha$)}
\STATE $h=h+1$
\ENDWHILE
\ENDIF
\IF{$n \geq 2$}
\STATE \textbf{PartialBlockMid($i, \lambda_{i-1}, r$)}
\ENDIF
\STATE \textbf{ReverseBuild($0, i, \lambda_{i-1}, l$)}
\end{algorithmic}
\end{algorithm}


\begin{algorithm}
\caption{Determine $\{ \lambda_i \}$ for $\frac{3}{2} \alpha > \beta >\alpha$, $(\beta-\alpha)|\alpha$, $n \geq \frac{\alpha}{\beta-\alpha} +1$}
\label{alg:closedivides}
\begin{algorithmic}
\STATE $l=\beta-\alpha$
\STATE $c = \lceil \frac{\alpha}{l} \rceil = \frac{\alpha}{l}$
\STATE $d = \alpha  \text{mod} l = 0$
\STATE $\lambda_0 = n\beta+\alpha-1$
\STATE $i=1$
\STATE \textbf{Build($c-2, i, \lambda_{i-1}$)}
\STATE $h=1$
\WHILE{$h \leq nl-\alpha+l$}
\STATE \textbf{BlockClose($i, \lambda_{i-1}, c, d, l, \alpha$)}
\STATE $h=h+1$
\ENDWHILE
\STATE \textbf{ReverseBuildPartial($c-2, i, \lambda_{i-1}, l$)}
\STATE \textbf{ReverseBuild($c-3, i, \lambda_{i-1}, l$)}
\end{algorithmic}
\end{algorithm}


\begin{algorithm}
\caption{Determine $\{ \lambda_i \}$ for $\frac{3}{2} \alpha > \beta >\alpha$, $(\beta-\alpha) \nmid \alpha$, $n \geq \lceil \frac{\alpha}{\beta-\alpha} \rceil+1$}
\label{alg:closedoesnotdivide}
\begin{algorithmic}
\STATE $l=\beta-\alpha$
\STATE $c = \lceil \frac{\alpha}{l} \rceil$
\STATE $d = \alpha  \text{mod} l$
\STATE $\lambda_0 = n\beta+\alpha-1$
\STATE $i=1$
\STATE \textbf{Build($c-2, i, \lambda_{i-1}$)}
\STATE $h=1$
\WHILE{$h \leq n-c$}
\STATE \textbf{BlockClose($i, \lambda_{i-1}, c, d, l, \alpha$)}
\STATE $h=h+1$
\ENDWHILE
\STATE \textbf{PartialBlockClose($i, \lambda_{i-1}, c, d$)}
\STATE \textbf{ReverseBuildPartial($c-2, i, \lambda_{i-1}, l$)}
\STATE \textbf{ReverseBuild($c-3, i, \lambda_{i-1}, l$)}
\end{algorithmic}
\end{algorithm}


\begin{algorithm}[H]
\caption{Determine $\{ \lambda_i \}$ for $\frac{3}{2} \alpha > \beta >\alpha $, $2 \leq n < \lceil \frac{\alpha}{l}\rceil +1$} 
\label{alg:closesmalln}
\begin{algorithmic}
\STATE $l = \beta-\alpha$
\STATE $\lambda_0 = n\beta+\alpha-1$
\STATE $i=1$
\STATE \textbf{Build($n-2, i, \lambda_{i-1}$)}
\STATE $h=1$
\WHILE{$h \leq \beta-nl$}
\STATE \textbf{onestwo}$(n-1,i,\lambda_{i-1})$
\STATE $h=h+1$
\ENDWHILE
\STATE \textbf{ReverseBuildPartial($n-2, i, \lambda_{i-1}, l$)}
\IF{$n \geq 3$}
\STATE \textbf{ReverseBuild($n-3, i, \lambda_{i-1}, l$)}
\ENDIF
\end{algorithmic}
\end{algorithm}


\begin{algorithm}
\caption{Determine $\{ \lambda_i \}$ for $\alpha=\beta$, $n \geq 1$}
\label{alg:equal}
\begin{algorithmic}
\STATE $i=1$
\STATE $\lambda_0 = (n+1)\alpha-1$
\STATE $h=1$
\WHILE{$h \leq \alpha-1$}
\STATE \textbf{onestwo}$(n-1,i,\lambda_{i-1})$
\STATE $h=h+1$
\ENDWHILE
\STATE \textbf{PartialBlockEqual($i, \lambda_{i-1}, n$)}
\end{algorithmic}
\end{algorithm}

\textbf{Sub-routines for Remaining Algorithms}


\begin{multicols}{2}

\begin{algorithmic}
\STATE \textbf{onestwo($x,i, \lambda_{i-1}$)}
\STATE $t=1$
\WHILE{$t \leq x$}
\STATE $\lambda_i = \lambda_{i-1}-1$
\STATE $t=t+1$
\STATE $i=i+1$
\ENDWHILE
\STATE $\lambda_i = \lambda_{i-1}-2$
\STATE $i=i+1$
\STATE \textbf{RETURN}
\end{algorithmic}

\columnbreak 


\begin{algorithmic}
\STATE \textbf{revonestwo($x,i, \lambda_{i-1}$)}
\STATE $\lambda_i=\lambda_{i-1}-2$
\STATE $i=i+1$ 
\STATE$t=1$
\WHILE{$t \leq x$}
\STATE $\lambda_i = \lambda_{i-1}-1$
\STATE $t=t+1$
\STATE $i=i+1$
\ENDWHILE
\STATE \textbf{RETURN}
\end{algorithmic}

\end{multicols}

\vspace{0.1in}

\begin{multicols}{2}


\begin{algorithmic}
\STATE \textbf{Build($limq, i, \lambda_{i-1}$)}
\STATE $q=0$
\WHILE{$q \leq limq$}
\STATE $j=1$
\WHILE{$j \leq l$}
\STATE \textbf{onestwo}$(q,i, \lambda_{i-1})$
\STATE $j=j+1$
\ENDWHILE
\STATE $q=q+1$
\ENDWHILE
\STATE \textbf{RETURN}
\end{algorithmic}

\columnbreak

\begin{algorithmic}
\STATE \textbf{ReverseBuild($limq, i, \lambda_{i-1}, l$)}
\STATE $q = limq$
\WHILE{$q \geq 0$}
\STATE $j=1$
\WHILE{$j \leq l$}
\STATE \textbf{revonestwo}$(q,i,\lambda_{i-1})$
\STATE $j=j+1$
\ENDWHILE
\STATE $q=q-1$
\ENDWHILE
\STATE \textbf{RETURN}
\end{algorithmic}

\end{multicols}


\begin{algorithmic}
\STATE \textbf{ReverseBuildPartial($limq, i, \lambda_{i-1}, l$)}
\STATE $j=1$
\WHILE{$j \leq limq$}
\STATE $\lambda_i = \lambda_{i-1}-1$
\STATE $j=j+1$
\ENDWHILE
\STATE $j=2$
\WHILE{$j \leq l$}
\STATE \textbf{revonestwo}$(limq,i, \lambda_{i-1})$
\STATE $j=j+1$
\ENDWHILE
\STATE \textbf{RETURN}
\end{algorithmic}

\begin{multicols}{2}

\begin{algorithmic}
\STATE \textbf{BlockMid($i, \lambda_{i-1}, r, \alpha$)}
\STATE $t=1$
\WHILE{$t \leq 2r-1$}
\IF{$t$ is odd}
\STATE $\lambda_{i}=\lambda_{i-1}-1$
\ELSE[$t$ is even]
\STATE $\lambda_{i}=\lambda_{i-1}-2$
\ENDIF
\STATE $t=t+1$
\STATE $i=i+1$
\ENDWHILE
\STATE $t=1$ 
\WHILE{$t \leq \alpha - (2r-1) = \beta-3\alpha+1$}
\STATE$\lambda_{i}=\lambda_{i-1}-2$
\STATE $t=t+1$
\STATE $i=i+1$
\ENDWHILE
\STATE \textbf{RETURN}
\end{algorithmic}

\columnbreak
\begin{algorithmic}
\STATE \textbf{PartialBlockMid($i, \lambda_{i-1}, r$)}
\STATE $t=1$
\WHILE{$t \leq 2r-1$}
\IF{$t$ is odd}
\STATE $\lambda_{i}=\lambda_{i-1}-1$
\ELSE[$t$ is even]
\STATE $\lambda_{i}=\lambda_{i-1}-2$
\ENDIF
\STATE $t=t+1$
\STATE $i=i+1$
\ENDWHILE
\STATE \textbf{RETURN}
\end{algorithmic}

\end{multicols}

\begin{multicols}{2}


\vspace{0.2in}
\begin{algorithmic}
\STATE \textbf{BlockClose($i, \lambda_{i-1}, c, d, l, \alpha$)}
\IF{$l | \alpha$}
\STATE \textbf{onestwo}$(c-1,i, \lambda_{i-1})$ 
\ELSE
\STATE $j=1$
\WHILE{$j \leq d$}
\STATE \textbf{onestwo}$(c-1,i,\lambda_{i-1})$
\STATE $j=j+1$
\ENDWHILE
\WHILE{$j \leq l$}
\STATE \textbf{onestwo}$(c-2,i,\lambda_{i-1})$
\STATE $j=j+1$
\ENDWHILE
\ENDIF
\STATE \textbf{RETURN}
\end{algorithmic}

\columnbreak

\begin{algorithmic}
\STATE \textbf{PartialBlockClose($i, \lambda_{i-1}, c, d$)}
\STATE $j=1$
\WHILE{$j \leq d$}
\STATE \textbf{onestwo}$(c-1,i,\lambda_{i-1})$
\STATE $j=j+1$
\ENDWHILE
\STATE \textbf{RETURN}
\end{algorithmic}

\end{multicols}


\begin{algorithmic}
\STATE \textbf{PartialBlockEqual($i, \lambda_{i-1}, n$)}
\STATE $h=1$
\WHILE{$h\leq n-1$}
\STATE $\lambda_i=\lambda_{i-1}-1$
\STATE $i=i+1$
\STATE $h=h+1$
\ENDWHILE
\STATE \textbf{RETURN}
\end{algorithmic}


\subsection{Description of the Algorithms}
\label{sec:notesonalgs}

We will call the $\lambda_i$s produced by these algorithms the \textit{proposed invariants} of $\text{gin}(I^n)$.  Each of the algorithms can be divided into the following three stages:
\begin{enumerate}
\item  the Build (absent in the cases where $\beta \geq 2\alpha-1$ and $\alpha = \beta$);
\item the Pattern (consists of full or partial repetitions of a Pattern Block\footnote{As their names suggest, \textbf{BlockFar}, \textbf{BlockMid}, and \textbf{BlockClose} are Pattern Blocks.  In the cases where $\alpha=\beta$ and $\frac{3}{2} > \beta> \alpha$, $2 \leq n < \lceil \frac{\alpha}{l} \rceil +1$, the Pattern Block is simply the \textbf{onestwo} subroutine.}); and
\item the Reverse Build (also absent in the cases where $\beta \geq 2\alpha-1$ and $\alpha = \beta$).  
\end{enumerate}
It will be convenient to divide the proposed invariants produced by an algorithm into the same three categories; for example, a $\lambda_i$ produced by the Build stage of an algorithm will be said to be part of the Build.

Each of the algorithms begins with defining $\lambda_0 = n\beta+\alpha-1$ (see Theorem \ref{thm:ginstructure}).  The other invariants are obtained one-by-one by subtracting 1, 2, or $\beta-2\alpha+2$ from the previous invariant in the sequence.  Patterns in the sequences $\{\lambda_0, \dots, \lambda_{k-1}\}$ emerge by looking at the sequences of gaps between the $\lambda_i$s;  thus, we set $g_i$ to be equal to the number subtracted from $\lambda_{i-1}$ to obtain $\lambda_i$, or $g_i := \lambda_{i-1}-\lambda_{i}$.   The sequence $\{g_i\}$ will be called the \textit{gap sequence} corresponding to the sequence of proposed invariants; note that this sequence will consist entirely of the numbers 1, 2, and $\beta-2\alpha+2$. 

Observe the following:
\begin{itemize}
\item Since all of the numbers $g_i$ are greater than 0, the sequences $\{\lambda_i\}$ produced by the algorithms are strictly decreasing.
\item The gap sequence of the Build written backwards generally gives the gap sequence of the Reverse Build.  The exception to this is in the algorithms corresponding to the cases where $\frac{3}{2} \alpha > \beta >\alpha$.  In these cases, everything but the final gap of the Build is reflected in the Reverse Build; this is why it is necessary to include the \textbf{ReverseBuildPartial} subroutine.
\item  The gap sequences of the Build, the Reverse Build, and the Pattern Blocks, are independent of $n$ except in Algorithm \ref{alg:closesmalln} corresponding to the case where $\frac{3}{2}\alpha > \beta >\alpha$, $2 \leq n < \lceil \frac{\alpha}{l} \rceil+1$ and in Algorithm \ref{alg:equal} corresponding to the case where $\alpha=\beta$.  The only part of the other algorithms that changes as $n$ increases is the number of times that the Pattern Block is repeated.
\item The last $\lambda_i$ produced by the algorithms is $\lambda_{k-1} = \beta-\alpha+1$.  Note that, by Theorem \ref{thm:ginstructure}, this condition must be satisfied for the algorithms to produce the invariants of $\text{gin}(I^n)$.  We can check that this condition holds by showing that 
\begin{eqnarray*}
\sum_{i=1}^{k-1} g_i &=& \lambda_0 - \lambda_{k-1} \\
&=& (n\beta+\alpha-1)-(\beta-\alpha+1) \\
&=& (n-1)\beta+2\alpha-2.
\end{eqnarray*}
\item The conditions on $\alpha$ and $\beta$ ensure that the algorithms make sense.  For example, in Algorithm \ref{alg:closedivides} \textbf{ReverseBuild}($c-3, \dots$) is well-defined because when $\frac{3}{2} \alpha > \beta > \alpha$, 
$$c = \frac{\alpha}{l} > \frac{\alpha}{3\alpha/2-\alpha} = 2.$$
\end{itemize}

We encourage the reader to consult Appendix \ref{ap:examples} to get a better feel for the sequences $\{g_i\}$ and $\{\lambda_i\}$ produced by these algorithms.  It contains examples of the outputs for fixed $\alpha$, $\beta$, and $n$ and illustrations that indicate what happens for a general $n$ in most cases.

\section{Proof of Theorem \ref{thm:mainthm}}
\label{sec:proofofmain}

In this section we will prove Theorem \ref{thm:mainthm}.  The proof is divided into six parts, one for each of the cases and algorithms referred to in the theorem.  The proof in each case will involve the following steps:

\begin{enumerate}
\item write non-recursive formulas for the proposed invariants $\lambda_i$ produced by the algorithm;
\item compute $H_J(t)$ where 
$$J=(x^k, x^{k-1}y^{\lambda_{k-1}}, \dots, xy^{\lambda_1}, y^{\lambda_0})$$
and the invariants $\lambda_i$ are as above; and
\item rewrite $H_{I^n}(t)$ in an appropriate form, sometimes using the assumptions on $\alpha$ and $\beta$ for the particular case, and simplify the expression to show that it is equal to $H_{J}(t)$.  By Lemma \ref{lem:determineinvariants} this will prove that $J = \text{gin}(I^n)$ so that the invariants produced by the algorithm are the invariants of $\text{gin}(I^n)$.
\end{enumerate}

Since the required calculations are routine and long, details are left to the Appendices~\ref{ap:formulas} and \ref{ap:calculations}.  In particular, Appendix~\ref{ap:formulas} contains the derivations of the non-recursive formulas for the $\lambda_i$ from the algorithms while Appendix~\ref{ap:calculations} contains details of the $H_{I^n}(t)$ calculation and simplifications of the partial sums $\sum_i {t-\lambda_i-i+m-2 \choose m-2}$ that appear in the expression for $H_J(t)$ from Section~\ref{sec:hilbmethod}.

For convenience, we will divide the formulas and long calculations into parts according to whether they involve invariants and indexing from the Build, Pattern, or Reverse Build of the algorithm as described in Section~\ref{sec:notesonalgs}.

As before, $l=\beta-\alpha$ and $g_i = \lambda_{i-1}-\lambda_i$.

\subsection{The Case $\beta \geq 2\alpha-1$,\,\,\,$n \geq 1$}

\subsubsection{Formulas for the proposed invariants}
\label{sec:largeformulas}
First we write a closed form expression for the $\lambda_v$ produced by Algorithm \ref{alg:far}.  For details on how this formula follows from the algorithm, see Section \ref{apsec:farformulas} of Appendix \ref{ap:formulas}.

If $v = j\alpha+s$ where $s = 0, \dots, \alpha-1$ and $j=0, \dots, n-1$,
$$ \lambda_v =  (n-j)\beta+\alpha-1-2s.$$

\subsubsection{The Hilbert function of $J$}
Suppose that $J=(x^k, x^{k-1}y^{\lambda_{k-1}}, \dots, xy^{\lambda_1}, y^{\lambda_0})$ where the $\lambda_v$ are given by the formula in Section \ref{sec:largeformulas}.  Then, by Proposition \ref{prop:HilbofJ},  the Hilbert function for $H_J(t)$ is 

\begin{eqnarray*}
H_{J} (t) &=& \sum_{i=0}^{k-1} {t-\lambda_i-i+(m\!\!-\!\!2) \choose (m\!\!-\!\!2)} + {t-k+(m\!\!-\!\!1) \choose (m\!\!-\!\!1)}\\
&=& \sum_{j=0}^{n-1} \sum_{s=0}^{\alpha-1} {t-[(n-j)\beta+\alpha-1-2s]-[j\alpha+s]+(m\!\!-\!\!2) \choose (m\!\!-\!\!2)} + {t-n\alpha+(m\!\!-\!\!1) \choose (m\!\!-\!\!1)}\\
&=& \sum_{j=0}^{n-1} \sum_{s=0}^{\alpha-1} {t-(n-j)\beta-\alpha-j\alpha+s+m-1 \choose m-2} + {t-n\alpha+m-1 \choose m-1} \\
&=& \sum_{j=0}^{n-1} \Bigg[ {t-(n-j)\beta-\alpha-j\alpha+ \alpha + m-1 \choose m-1} - {t-(n-j)\beta-\alpha-j\alpha+m-1 \choose m-1} \Bigg] \\
&&+ {t-n\alpha+m-1 \choose m-1}\\
&=&\sum_{p=1}^n {t-p\beta-(n-p)\alpha+m-1 \choose m-1} - \sum_{q=1}^n {t-(n-q+1)\beta-q\alpha+m-1 \choose m-1} \\
&&+ {t-n\alpha+m-1 \choose m-1}.\\
\end{eqnarray*}

\subsubsection{The Hilbert function of $I^n$}
By Proposition \ref{prop:HilbofIn},
\begin{eqnarray*}
H_{I^n}(t) &=&  \sum_{j=1}^n \Bigg ( {t-\alpha(n-j)-\beta j +(m\!\!-\!\!1) \choose (m\!\!-\!\!1)} - {t-\alpha j -\beta (n+1 -j)+(m\!\!-\!\!1)\choose (m\!\!-\!\!1)} \Bigg) \\
{} {} {} {} {}  & & +{t-n\alpha+(m\!\!-\!\!1) \choose (m\!\!-\!\!1)}
\end{eqnarray*}
Since $H_J(t) = H_{I^n}(t)$, Lemma \ref{lem:determineinvariants} implies that $J=\text{gin}(I^n)$ and thus that the numbers produced  by Algorithm \ref{alg:far} are the invariants of $\text{gin}(I^n)$ when $\beta \geq 2\alpha-1$ and $n \geq 1$.

\subsection{The Case $2\alpha -1> \beta \geq \frac{3}{2} \alpha, n \geq 2$}

Throughout this section, we will set $r:=2\alpha-\beta>0$.

\subsubsection{Formulas for the proposed invariants}
\label{sec:midformulas}

First we will write closed-form expressions for the numbers $\lambda_i$ produced by Algorithm \ref{alg:mid}.  Details about how these formulas are be obtained from the algorithm can be found in Section \ref{apsec:midformulas} of Appendix \ref{ap:formulas}.  To match the work in the appendix, we distinguish the formulas for invariants produced by the Build, the Reverse Build, and the Pattern phases of the Algorithm \ref{alg:mid}.  

\noindent \textit{Formula for $\lambda_i$ in the Build.}
For $v = 0, \dots, l$,  
$$\lambda_v =  \alpha(n+1) + l\cdot n -1-2v.$$

\noindent \textit{Formulas for $\lambda_i$ in the Reverse Build.}
For $v=k-i = n\alpha-i$ where $i=1, \dots, l+1$,
$$\lambda_{v} =  l+1+2(i-1) = l+2i-1.$$

\noindent \textit{Formulas for $\lambda_i$ in the Pattern.}

 \begin{itemize}
\item[\ding{203}] For $v = l+j\alpha+y$ where $j=0, \dots, (n-3)$ and $y=2r-1, \dots, \alpha-1$,
$$ \lambda_v=\lambda_0 - [2l+(2\alpha-r)j+2y-r] = \lambda_0-[2l+(\alpha+l)j+2y-(\alpha-l)].$$
\item[\ding{204}] For $v = l+j\alpha$ where $j=1, \dots, n-2$,
$$\lambda_v = \lambda_0 - [j(2\alpha-r) +2l = \lambda_0-[j(l+\alpha)+2l]].\footnote{Note that when $n=2$, the ranges for $j$ in \ding{203} and \ding{204} are empty.  This reflects the fact that Algorithm \ref{alg:mid} only includes the Partial Pattern Block when $n=2$.
}$$
\item[\ding{205}] For $ v= l+j\alpha+2p$ where $j=0, \dots, n-2$ and $p = 1, \dots, r-1 = \alpha-l-1$,
$$\lambda_v = \lambda_0 - [2l+(\alpha+l) j + 2p + p].$$
\item[\ding{206}] For $v = l+j\alpha+2p-1$ where $j = 0, \dots, (n-2)$ and $p = 1,\dots, r-1$, 
$$\lambda_v = \lambda_0 - [2l+(\alpha+l)j+2p-2 + p]$$
\end{itemize}

\subsubsection{The Hilbert function of $J$}
\label{sec:midsum}

Suppose that $J=(x^k, x^{k-1}y^{\lambda_{k-1}}, \dots, xy^{\lambda_1}, y^{\lambda_0})$ where the $\lambda_i$ are the invariants produced by Algorithm \ref{alg:mid} and are given by the formulas in Section \ref{sec:largeformulas}.  By Proposition \ref{prop:HilbofJ}
$$H_J(t) = \sum_{i=0}^{k-1} {t-\lambda_i-i+(m\!\!-\!\!2) \choose (m\!\!-\!\!2)} + {t-k+(m\!\!-\!\!1) \choose (m\!\!-\!\!1)}.$$ 
  Set 
 $$X_j := t-n\alpha-jl$$ 
 and 
 $$Y_j := t-(n+1)\alpha-jl.$$  
In Section \ref{apsec:midcalculations} of Appendix \ref{ap:calculations}, we simplify the partial sums $\sum_i {t-\lambda_i-i+m-2 \choose m-2}$ as $i$ ranges over the Build, Reverse Build, and Pattern.  Adding these together with $ {t-n\alpha+ m-1\choose m-1}$ we get the following expression for $H_J(t)$ when $n \geq 3$.
 
\begin{eqnarray*}
 H_{J}(t) &=&  \sum_{j=2}^n {X_j +m\!\!-\!\!1 \choose m\!\!-\!\!1} + \sum_{j=2}^n{X_j+m\!\!-\!\!2 \choose m\!\!-\!\!1} - \sum_{j=3}^n {X_j +m\!\!-\!\!2 \choose m\!\!-\!\!1}-\sum_{j=1}^{n-1} {Y_j+m\choose m\!\!-\!\!1}\\ 
 & &- \sum_{j=1}^{n-1} {Y_j+m\!\!-\!\!1 \choose m\!\!-\!\!1}+ \sum_{j=1}^{n-2} {Y_j+m\!\!-\!\!1 \choose m\!\!-\!\!1} + \sum_{j=1}^{n-2} {Y_j+m\!\!-\!\!1 \choose m\!\!-\!\!2}+{Y_{n-1}+m \choose m\!\!-\!\!1} \\
 &&- {Y_n+m\!\!-\!\!1 \choose m\!\!-\!\!1}  + {X_1+m\!\!-\!\!1\choose m\!\!-\!\!1} - {X_2+m\!\!-\!\!2 \choose m\!\!-\!\!1} + {t-n\alpha+ m\!\!-\!\!1\choose m\!\!-\!\!1} \\
&=& \sum_{j=2}^n {X_j +m\!\!-\!\!1 \choose m\!\!-\!\!1} + {X_2+m\!\!-\!\!2 \choose m\!\!-\!\!1} + \Bigg( \sum_{j=1}^{n-2} {Y_j+m\!\!-\!\!1 \choose m-2}-\sum_{j=1}^{n-1} {Y_j+m\choose m\!\!-\!\!1} \Bigg)  \\
& & -{Y_{n-1}+m\!\!-\!\!1 \choose m\!\!-\!\!1} +{Y_{n-1}+m \choose m\!\!-\!\!1} - {Y_n+m\!\!-\!\!1 \choose m\!\!-\!\!1}+  {X_1+m\!\!-\!\!1\choose m\!\!-\!\!1} - \\
& & {X_2+m\!\!-\!\!2 \choose m\!\!-\!\!1} + {t-n\alpha+ m\!\!-\!\!1\choose m\!\!-\!\!1} \\
&=& \sum_{j=2}^n {X_j +m\!\!-\!\!1 \choose m\!\!-\!\!1}  - \sum_{j=1}^{n-2} {Y_j+m\!\!-\!\!1\choose m\!\!-\!\!1} -{Y_{n-1}+m \choose m\!\!-\!\!1} - {Y_{n-1} + m\!\!-\!\!1 \choose m\!\!-\!\!1}\\
& & +{Y_{n-1}+m \choose m\!\!-\!\!1} - {Y_n+m\!\!-\!\!1 \choose m\!\!-\!\!1}+  {X_1+m\!\!-\!\!1\choose m\!\!-\!\!1}  + {t-n\alpha+ m\!\!-\!\!1\choose m\!\!-\!\!1} \\
&=& \sum_{j=1}^n {X_j +m\!\!-\!\!1 \choose m\!\!-\!\!1} - \sum_{j=1}^{n} {Y_j+m\!\!-\!\!1\choose m\!\!-\!\!1} + {t-n\alpha+ m\!\!-\!\!1\choose m\!\!-\!\!1}.  \\
\end{eqnarray*}

When $n=2$,
\begin{eqnarray*}
H_J(t) &=& {X_2+m-1 \choose m-1} - {Y_1+m \choose m-1} + {X_2+m-2 \choose m-1} - {Y_1+m-1 \choose m-1} \\
& & + {Y_{n-1} + m \choose m-1} - {Y_n + m-1 \choose m-1} + {X_1 + m-1 \choose m-1} - {X_2 +m-2 \choose m-1}\\
&=& {X_2 +m-1 \choose m-1} - {Y_1 +m-1 \choose m-1} - {Y_2 +m-1 \choose m-1} + {X_1 +m-1 \choose m-1}.\\
\end{eqnarray*}

\subsubsection{The Hilbert function of $I^n$}

By Proposition \ref{prop:HilbofIn},

\begin{eqnarray*}
H_{I^n}(t)&=&\sum_{j=1}^n \Bigg( {t-\alpha n -jl +(m\!\!-\!\!1) \choose (m\!\!-\!\!1)} - {t-\alpha(n+1) - lj + (m\!\!-\!\!1) \choose (m\!\!-\!\!1)}\Bigg)\\
&&+ {t - n\alpha+m-1 \choose m-1}.
\end{eqnarray*}

Thus, $H_{I^n}(t) = H_J(t)$ so, by Lemma \ref{lem:determineinvariants}, $J= \text{gin}(I^n)$.  Therefore, Algorithm \ref{alg:mid} produces the invariants of $\text{gin}(I^n)$ when $\beta \geq 2\alpha-1$ and $n \geq 1$.


\subsection{The Case $\frac{3}{2} \alpha > \beta > \alpha$, $l | \alpha$, $n\geq \frac{\alpha}{l}+1$}
\label{sec:divides}

Throughout this section, set $c:=\alpha/l$; this is an integer by assumption.

\subsubsection{Formulas for the proposed invariants}
\label{sec:dividesformulas}
As in previous cases, in this section we write closed form expressions for the numbers $\lambda_i$ produced by Algorithm \ref{alg:closedivides}.  See Section \ref{apsec:dividesformulas} of Appendix \ref{ap:formulas} for details on how the formulas stated here were obtained.  To be consistent with the work there, the formulas coming from each of the three phases of the algorithm (the Build, Reverse Build, and Pattern) are recorded separately.

\textit{Formulas for $\lambda_i$ the Build.}
\begin{itemize}
\item[\ding{202}] For $v=0, \dots, l$, 
$$\lambda_v = \lambda_0-2v.$$
\item[\ding{203}] For $v=l(1+\cdots+q) +(q+1)j$ where $q=1, \dots, (c-2)$, $j=1, \dots, l$, 
$$\lambda_v = \lambda_0-[(2+\cdots+(q+1))l+(q+1)j+j]$$
\item[\ding{204}] For $v=l(1+\cdots + q) +(q+1)j -x$ where $q=1, \dots, (c-2)$, $j=1, \dots, l$, $x = 1,\dots, q$, 
$$\lambda_v=\lambda_0-[(2+\cdots+(q+1))l+(q+1)j-x+j-1]$$
\end{itemize}

\textit{Formulas for $\lambda_i$ in the Reverse Build.}
\begin{itemize}
\item[\ding{202}]  For $v=(k-1)-j$ where  $j=0, \dots, l$, 
$$\lambda_v = \lambda_{k-1}+2j.$$
\item[\ding{203}] For $v=(k-1)-(l(1+\cdots+q) +(q+1)j)$ where $q=1, \dots, (c-2)$, $j=1, \dots, l$,
$$\lambda_v = \lambda_{k-1}+[(2+\cdots+(q+1))l+(q+1)j+j].$$
However, when $q=c-2$ and $j=l$, $\lambda_v$ is in the Pattern, not in the Reverse Build.
\item[\ding{204}] For $v=(k-1)-(l(1+\cdots + q) +(q+1)j -x)$ where $q=1, \dots, (c-2)$, $j=1, \dots, l$, $x = 1,\dots, q$, 
$$\lambda_v=\lambda_{k-1}+[(2+\cdots+(q+1))l+(q+1)j-x+j-1].$$
However, when $q=c-2$, $j=l$, and $x=1$, $\lambda_v$ is in the Pattern, not the Reverse Build.
\end{itemize}

\textit{Formulas for $\lambda_i$ in the Pattern.}

 Let 
$$E:= l(1+\cdots+(c-1))$$
and
$$B:=l(2+\cdots + c).$$

\begin{itemize}
\item[\ding{202}]  For $v=E+jc+i$ where $j=0, \dots, ln-\alpha+l-1$, $i=1, \dots, c-1$,
$$\lambda_v = \lambda_0-B-[j(c+1)+i]$$
\item[\ding{203}] For $v = E+jc$ where $j=1, \dots, ln-\alpha+l-1$,
$$\lambda_v = \lambda_0-B-[j(c+1)]$$
\end{itemize}

\subsubsection{The Hilbert function of $J$}
\label{sec:dividessum}

As before we set $$X_j = t-n\alpha-jl$$ and $$Y_j = t-(n+1)\alpha-jl.$$ 
 
Note that $X_c = Y_0$ and $Y_{n-c+1}=X_{n+1}$:
$$X_c = t-n\alpha-cl=t-(n+1)\alpha = Y_0$$
$$Y_{n-c+1} = t-(n+1)\alpha-nl+lc-l = t-n\alpha-l(n+1) = X_{n+1}.$$

We will apply these relations in the calculations below.

Let $J = (x^k, x^{k-1}y^{\lambda_{k-1}}, \dots, xy^{\lambda_1}, y^{\lambda_0})$ where the $\lambda_i$ are the numbers produced by Algorithm  \ref{alg:closedivides}  and are given by the formulas in Section \ref{sec:dividesformulas}.  To compute $H_J(t)$ we use the formula from Proposition \ref{prop:HilbofJ}.

Section \ref{apsec:dividescalculations} of Appendix \ref{ap:calculations}  contains expressions for the partial sums $\sum_i {t-\lambda_i-i+m-2 \choose m-2}$ as $i$ ranges over the Build, the Reverse Build, and the Pattern.  Adding these together with ${t-n\alpha +m-1 \choose m-1}$, we obtain the following expression for $H_J(t)$.

\begin{eqnarray*}
 H_{J}(t) &=&  {Y_{n-1}+m \choose m\!\!-\!\!1} -{Y_n +m\!\!-\!\!1 \choose m\!\!-\!\!1} +{Y_{n-c+1} + m \choose m\!\!-\!\!1}-{Y_{n-1}+m \choose m\!\!-\!\!1} \\
 &&+ (c-2){Y_{n-c+1}+m\!\!-\!\!1 \choose m\!\!-\!\!1} - \sum_{j=n-c+2}^{n-1} {Y_j+m\!\!-\!\!1 \choose m\!\!-\!\!1} + {X_1+ m\!\!-\!\!1 \choose m\!\!-\!\!1} - {X_2+m-2 \choose m\!\!-\!\!1}   \\
 & & +{X_2+m-2 \choose m\!\!-\!\!1} - {X_c+m-2 \choose m\!\!-\!\!1} +  \sum_{j=2}^{c-1} {X_j +m\!\!-\!\!1 \choose m\!\!-\!\!1} -(c-2) {X_c+m\!\!-\!\!1 \choose m\!\!-\!\!1}\\
 & & + (c-1) \Bigg[{Y_0+m\!\!-\!\!1 \choose m\!\!-\!\!1} -{X_{n+1}+m\!\!-\!\!1 \choose m\!\!-\!\!1} \Bigg] + {Y_0+m \choose m\!\!-\!\!1} -{X_{n+1}+m  \choose m\!\!-\!\!1}\\
 & & -{Y_0+m-2 \choose m-2} -{Y_0+m\!\!-\!\!1 \choose m-2}+{t-n\alpha+m\!\!-\!\!1 \choose m\!\!-\!\!1}\\
 &=&-\sum_{j=n-c+2}^n {Y_j+m\!\!-\!\!1 \choose m\!\!-\!\!1} + {Y_{n-c+1} + m \choose m\!\!-\!\!1}-{X_{n+1}+m \choose m\!\!-\!\!1} + (c-2) {Y_{n-c+1} +m\!\!-\!\!1 \choose m\!\!-\!\!1} \\
 & & -(c-1) {X_{n+1} + m\!\!-\!\!1 \choose m\!\!-\!\!1} + {Y_0+m\!\!-\!\!1 \choose m\!\!-\!\!1} + (c-1){Y_0 + m\!\!-\!\!1 \choose m\!\!-\!\!1}\\
 & &  - {Y_0+ m\!\!-\!\!2 \choose m\!\!-\!\!2} -  (c-2) {Y_0 +m\!\!-\!\!1 \choose m\!\!-\!\!1} - {X_c+ m\!\!-\!\!2 \choose m\!\!-\!\!1} \\
 & & + \sum_{j=1}^{c-1} {X_j+m\!\!-\!\!1 \choose m\!\!-\!\!1} + {t-n\alpha +m\!\!-\!\!1 \choose m\!\!-\!\!1}\\
 & =& -\sum_{j=n-c+2}^n {Y_j+m\!\!-\!\!1 \choose m\!\!-\!\!1} - {Y_{n-c+1} + m\!\!-\!\!1 \choose m\!\!-\!\!1}+ c{Y_0+m\!\!-\!\!1 \choose m\!\!-\!\!1}-(c-2) {Y_0 + m\!\!-\!\!1\choose m\!\!-\!\!1} \\
 & & - {Y_0 + m\!\!-\!\!2 \choose m\!\!-\!\!2} - {Y_0 + m\!\!-\!\!2 \choose m\!\!-\!\!1}  + \sum_{j=1}^{c-1} {X_j+m\!\!-\!\!1 \choose m\!\!-\!\!1}+ {t-n\alpha +m\!\!-\!\!1 \choose m\!\!-\!\!1}\\
 &=& -\sum_{j=n-c+1}^n {Y_j+m\!\!-\!\!1 \choose m\!\!-\!\!1} + 2{Y_0+ m\!\!-\!\!1 \choose m\!\!-\!\!1} - {Y_0 + m\!\!-\!\!1 \choose m\!\!-\!\!1} \\
 & & + \sum_{j=1}^{c-1} {X_j+m\!\!-\!\!1 \choose m\!\!-\!\!1} + {t-n\alpha +m\!\!-\!\!1 \choose m\!\!-\!\!1}\\
 &=& -\sum_{j=n-c+1}^n {Y_j+m\!\!-\!\!1 \choose m\!\!-\!\!1}+ \sum_{j=1}^{c} {X_j+m\!\!-\!\!1 \choose m\!\!-\!\!1} + {t-n\alpha +m\!\!-\!\!1 \choose m\!\!-\!\!1}\\
 \end{eqnarray*}

\subsubsection{The Hilbert function of $I^n$}
\label{sec:dividesrewrite}

 We may take advantage of the assumptions that $\frac{3}{2} \alpha > \beta > \alpha$ and $l | \alpha$ to rewrite the expression of $H_{I^n}(t)$ from Proposition \ref{prop:HilbofIn} as follows:

$$H_{I^n}(t) =  \sum_{j=1}^{c} {X_j+m-1\choose m-1} -\sum_{j=n-c+1}^{n} {Y_j+m-1 \choose m-1}+{t-n\alpha+m-1 \choose m-1}.$$

For the details of this calculation see Section \ref{apsec:dividescalculations} of Appendix~\ref{ap:calculations}.

Since $H_{I^n}(t) = H_J(t)$, Lemma \ref{lem:determineinvariants} implies that $\text{gin}(I^n) = J$ and that the numbers produced by Algorithm \ref{alg:closedivides} are the invariants of $\text{gin}(I^n)$ when $\frac{3}{2}\alpha > \beta >\alpha$, $(\beta-\alpha) | \alpha$, and $n \geq \frac{\alpha}{\beta-\alpha} +1$.


\subsection{The Case $\frac{3}{2} \alpha > \beta > \alpha$, $l \nmid  \alpha$, $n \geq \lceil \frac{\alpha}{l} \rceil+1$}

Throughout this section, set $c:=\lceil \frac{\alpha}{l} \rceil$ and $d := \alpha \mod l$.  Then 
$$d=\alpha-l(c-1) = \alpha-lc+l$$
so 
$$lc = \alpha+l-d.$$

This relation will be used during the calculations in this section.  

\subsubsection{Formulas for the proposed invariants}
\label{sec:notdivideformulas}
In this section we present closed form expressions for the $\lambda_i$ produced by Algorithm \ref{alg:closedoesnotdivide}.  As before, formulas for the $\lambda_i$ produced by the Build, the Reverse Build, and the Pattern are recorded separately. Details of how these formulas were obtained from Algorithm \ref{alg:closedoesnotdivide} can be found in Section \ref{apsec:notdivideformulas} of Appendix \ref{ap:formulas}. 

\textit{Formulas for $\lambda_i$ in the Build and Reverse Build.}

The formulas for $\lambda_i$ produced by the Build and the Reverse Build are exactly the same as those in the Build and Reverse Build in the previous case; see Section \ref{sec:dividesformulas} for the formulas.

\textit{Formulas for $\lambda_i$ in the Pattern.}

Set $E = l(1+ 2+ \cdots + (c-1))$ and $B = l(2+3+ \cdots +c)$. 

\begin{itemize}
\item[\ding{202}]  For $v = E+p\alpha+jc+i$ where $p=0, \dots, n-c$, $j=0,\dots, d-1$, and $i=1, \dots, c-1$,
$$\lambda_v=\lambda_0-B-[p(l+\alpha)+j(c+1)+i]$$
\item[\ding{203}] For $v = E+p\alpha+jc$ where $p=0, \dots, n-c$, and $j=1, \dots, d$, 
$$\lambda_v = \lambda_0-B-[p(l+\alpha+j(c+1)]$$
\item[\ding{204}] For $v=E+p\alpha+dc+j(c-1)+i$ where $p=0, \dots, n-c-1$, $j=0, \dots, l-d-1$, and $i=1, \dots, c-2$
$$\lambda_v=\lambda_0-B-[p(l+\alpha)+d(c+1)+jc+i]$$
\item[\ding{205}] For $v=E+p\alpha+dc+j(c-1)$ where $p=0, \dots, n-c-1$, and $j=1, \dots, l-d$,
$$\lambda_v = \lambda_0-B-[p(l+\alpha)+d(c+1)+jc]$$
\end{itemize}

\subsubsection{The Hilbert function of $J$}
\label{sec:correctlnotdivide}

Set  $X_j := t-n\alpha-lj$ and $Y_j := t-(n+1)\alpha-lj$ as before; we also set $$Z_j := t-(n+1)\alpha-lj+d.$$ 
Note that 
$$Y_0=X_c = t-n\alpha-cl=t-n\alpha-\alpha-l+d=t-(n+1)\alpha-l+d = Z_1;$$
we will apply this identity to get the third equality below.

Let $J=(x^k, x^{k-1}y^{\lambda_{k-1}}, \dots, xy^{\lambda_1}, y^{\lambda_0})$ where the $\lambda_i$ are the numbers produced by Algorithm \ref{alg:closedoesnotdivide}.  We use the closed form expressions for the $\lambda_i$ given above in the formula for $H_J(t)$ from Proposition \ref{prop:HilbofJ}. Section \ref{apsec:notdividecalculations} of Appendix \ref{ap:calculations} contains expressions for the partial sums $\sum_i {t-\lambda_i-i+m-2 \choose m-2}$ as $i$ ranges over the Build, the Reverse Build, and the Pattern.  Adding these together with ${t+n\alpha+m-1 \choose m-1}$, we obtain the following expression for $H_J(t)$.

\begin{eqnarray*}
 H_{J}(t) &=&  {X_1+ m\!\!-\!\!1  \choose  m\!\!-\!\!1 }-{X_2+ m\!\!-\!\!2  \choose  m\!\!-\!\!1 } + {X_2+m\!\!-\!\!2  \choose  m\!\!-\!\!1 }-{X_c+m\!\!-\!\!2  \choose  m\!\!-\!\!1 }+\sum_{j=2}^{c-1} {X_j+ m\!\!-\!\!1  \choose  m\!\!-\!\!1 } \\
 &&-(c-2){X_c+ m\!\!-\!\!1  \choose  m\!\!-\!\!1 }+{Y_{n-1}+m \choose  m\!\!-\!\!1 }-{Y_n+ m\!\!-\!\!1  \choose  m\!\!-\!\!1 }+{Y_{n-c+1}+m \choose  m\!\!-\!\!1 }-{Y_{n-1}+m \choose  m\!\!-\!\!1 } \\ 
 && +(c-2){Y_{n-c+1}+ m\!\!-\!\!1  \choose  m\!\!-\!\!1 }-\sum_{j=n-c+2}^{n-1} {Y_j+ m\!\!-\!\!1  \choose  m\!\!-\!\!1 }+(c-1)\sum_{j=1}^{n-c+1} {Z_j+ m\!\!-\!\!1  \choose  m\!\!-\!\!1 }\\
 & & +\sum_{j=1}^{n-c+1} {Z_j+m \choose  m\!\!-\!\!1 }-(c-2) \sum_{j=2}^{n-c+1} {Z_j+ m\!\!-\!\!1 \choose  m\!\!-\!\!1 } -\sum_{j=2}^{n-c+1} {Z_j+m \choose  m\!\!-\!\!1 } \\
 & &-(c-1)\sum_{j=1}^{n-c+1} {Y_j+ m\!\!-\!\!1 \choose  m\!\!-\!\!1 } -\sum_{j=1}^{n-c+1} {Y_j+m \choose  m\!\!-\!\!1 }\\
 & & +(c-2)\sum_{j=1}^{n-c} {Y_j+ m\!\!-\!\!1  \choose  m\!\!-\!\!1 } +\sum_{j=1}^{n-c} {Y_j+m \choose  m\!\!-\!\!1 } -{Z_1+m\!\!-\!\!2  \choose m\!\!-\!\!2 }\\
 & & -{Z_1+ m\!\!-\!\!1  \choose m\!\!-\!\!2 }+{t-n\alpha+ m\!\!-\!\!1  \choose  m\!\!-\!\!1 }\\
 &=&\sum_{j=1}^{c-1} {X_j+ m\!\!-\!\!1  \choose  m\!\!-\!\!1 }-{X_c+m\!\!-\!\!2  \choose  m\!\!-\!\!1 } -(c-2){X_c+ m\!\!-\!\!1  \choose  m\!\!-\!\!1 } \\
 & & -\sum_{j=n-c+2}^{n} {Y_j+ m\!\!-\!\!1  \choose  m\!\!-\!\!1 }+{Y_{n-c+1}+m \choose  m\!\!-\!\!1 }+(c-2){Y_{n-c+1}+ m\!\!-\!\!1  \choose  m\!\!-\!\!1 } \\
 & & -\sum_{j=1}^{n-c} {Y_j+ m\!\!-\!\!1  \choose  m\!\!-\!\!1 } -(c-1){Y_{n-c+1}+ m\!\!-\!\!1  \choose  m\!\!-\!\!1 }-{Y_{n-c+1}+m \choose  m\!\!-\!\!1 } \\
 & & +\sum_{j=2}^{n-c+1} {Z_j+ m\!\!-\!\!1  \choose  m\!\!-\!\!1 } +(c-1) {Z_1+ m\!\!-\!\!1  \choose  m\!\!-\!\!1 }+{Z_1+m \choose  m\!\!-\!\!1 } -{Z_1+m\!\!-\!\!2  \choose m\!\!-\!\!2 }\\
 & & -{Z_1+ m\!\!-\!\!1  \choose m\!\!-\!\!2 } + {t-n\alpha+ m\!\!-\!\!1  \choose  m\!\!-\!\!1 }\\
 &=& \sum_{j=1}^{c-1} {X_j+ m\!\!-\!\!1  \choose  m\!\!-\!\!1 }-\sum_{j=1}^{n} {Y_j+ m\!\!-\!\!1  \choose  m\!\!-\!\!1 }+\sum_{j=2}^{n-c+1} {Z_j+ m\!\!-\!\!1  \choose  m\!\!-\!\!1 }+{Z_1+ m\!\!-\!\!1  \choose  m\!\!-\!\!1 }\\
 &&-{Z_1+m\!\!-\!\!2  \choose  m\!\!-\!\!1 } +{Z_1+m \choose  m\!\!-\!\!1 }-{Z_1+m\!\!-\!\!2  \choose m\!\!-\!\!2 } -{Z_1+ m\!\!-\!\!1 \choose m\!\!-\!\!2 }+{t-n\alpha+ m\!\!-\!\!1  \choose  m\!\!-\!\!1 }\\
&=&\sum_{j=1}^{c-1} {X_j+ m\!\!-\!\!1  \choose  m\!\!-\!\!1 }-\sum_{j=1}^{n} {Y_j+ m\!\!-\!\!1  \choose  m\!\!-\!\!1 }+\sum_{j=1}^{n-c+1} {Z_j+ m\!\!-\!\!1  \choose  m\!\!-\!\!1 } + {t-n\alpha+ m\!\!-\!\!1  \choose  m\!\!-\!\!1 } \\
\end{eqnarray*}

\subsubsection{The Hilbert function of $I^n$}
\label{sec:notdividerewrite}

Using the assumption that $\frac{3}{2} \alpha > \beta > \alpha$ we may rewrite the Hilbert function of $I^n$ from Proposition \ref{prop:HilbofIn} in terms of $c$, $d$, $l$, and $n$.   As before, we set $X_j = t-n\alpha-lj$, $Y_j = t-(n+1)\alpha-lj$, and $Z_j := t-(n+1)\alpha-lj+d.$    Then
$$H_{I^n}(t) =  \sum_{j=1}^{c-1} {X_j+ m\!\!-\!\!1 \choose  m\!\!-\!\!1 } + \sum_{i=1}^{n-c+1} {Z_j+ m\!\!-\!\!1  \choose  m\!\!-\!\!1 } -\sum_{j=1}^n {Y_j+ m\!\!-\!\!1  \choose  m\!\!-\!\!1 }+{t-n\alpha+ m\!\!-\!\!1  \choose  m\!\!-\!\!1 }.$$
For details of this calculation see Section \ref{apsec:notdividecalculations} of Appendix~\ref{ap:calculations}.

Since $H_{I^n}(t) = H_J(t)$, we conclude by Lemma \ref{lem:determineinvariants} that $J= \text{gin}(I^n)$.  That is, the $\lambda_i$ produced by Algorithm \ref{alg:closedoesnotdivide} are the invariants of $\text{gin}(I^n)$ when $\frac{3}{2}\alpha > \beta >\alpha$, $(\beta-\alpha) \nmid \alpha$, and $n \geq \lceil \frac{\alpha}{\beta-\alpha} \rceil+1$. 


\subsection{The Case $\frac{3}{2} \alpha > \beta > \alpha$,  $n < \lceil \frac{\alpha}{l} \rceil +1$}

\subsubsection{Formulas for the proposed invariants}
\label{sec:closesmallnformulas}

In this section we present closed form expressions for the $\lambda_i$ produced by Algorithm \ref{alg:closesmalln}.  As before, formulas for the $\lambda_i$ arising from the Build, the Reverse Build, and the Pattern phases of the algorithm are recorded separately.  Details of the derivations of these formulas can be found in Section \ref{apsec:closesmallnformulas} of Appendix \ref{ap:formulas}.

\textit{Formulas for $\lambda_i$ the Build.}
\begin{itemize}
\item[\ding{202}] For $v=0, \dots, l$, 
$$\lambda_v = \lambda_0-2v.$$
\item[\ding{203}] For $v=l(1+\cdots+q) +(q+1)j$ where $q=1, \dots, (n-2)$, $j=1, \dots, l$, 
$$\lambda_v = \lambda_0-[(2+\cdots+(q+1))l+(q+1)j+j]$$
\item[\ding{204}] For $v=l(1+\cdots + q) +(q+1)j -x$ where $q=1, \dots, (n-2)$, $j=1, \dots, l$, $x = 1,\dots, q$, 
$$\lambda_v=\lambda_0-[(2+\cdots+(q+1))l+(q+1)j-x+j-1]$$
\end{itemize}

\textit{Formulas for $\lambda_i$ in the Reverse Build.}
\begin{itemize}
\item[\ding{202}]  For $v=(k-1)-j$ where  $j=0, \dots, l$, 
$$\lambda_v = \lambda_{k-1}+2j.$$
\item[\ding{203}] For $v=(k-1)-(l(1+\cdots+q) +(q+1)j)$ where $q=1, \dots, (n-2)$, $j=1, \dots, l$,
$$\lambda_v = \lambda_{k-1}+[(2+\cdots+(q+1))l+(q+1)j+j].$$
However, when $q=n-2$ and $j=l$, $\lambda_v$ is in the Pattern, not in the Reverse Build.
\item[\ding{204}] For $v=(k-1)-(l(1+\cdots + q) +(q+1)j -x)$ where $q=1, \dots, (n-2)$, $j=1, \dots, l$, $x = 1,\dots, q$, 
$$\lambda_v=\lambda_{k-1}+[(2+\cdots+(q+1))l+(q+1)j-x+j-1].$$
However, when $q=n-2$, $j=l$, and $x=1$, $\lambda_v$ is in the Pattern, not the Reverse Build.
\end{itemize}

\textit{Formulas for $\lambda_i$ in the Pattern.}

Let 
$$E:= l(1+\cdots+(n-1))$$
and
$$B:=l(2+\cdots + n).$$

\begin{itemize}
\item[\ding{202}]  For $v=E+jn+i$ where $j=0, \dots, \beta-nl-1$, $i=1, \dots, n-1$,
$$\lambda_v = \lambda_0-B-[j(n+1)+i].$$
\item[\ding{203}] For $v = E+jn$ where $j=1, \dots, \beta-nl-1$,
$$\lambda_v = \lambda_0-B-[j(n+1)].$$
\end{itemize}

\subsubsection{The Hilbert function of $J$}
\label{sec:closesmallnsum}

As before, we set $X_j = t-n\alpha-lj$ and $Y_j = t-(n+1)\alpha-lj$.  

Let $J= (x^k, x^{k-1}y^{\lambda_{k-1}}, \dots, xy^{\lambda_1}, y^{\lambda_0})$ where the $\lambda_i$ are given by the formulas above and are the invariants produced by Algorithm \ref{alg:closesmalln}.  In this section we will compute the Hilbert function $H_J(t)$ using the expression from Proposition \ref{prop:HilbofJ}.  Section \ref{apsec:closesmallncalculations} of Appendix \ref{ap:calculations} contains simplifications of the partial sums $\sum_i {t-\lambda_i-i+m-2 \choose m-2}$ as $i$ ranges over the Build, the Reverse Build, and the Pattern.  Adding these together with ${t-n\alpha+m-1 \choose m-1}$ we have:

\begin{eqnarray*}
 H_{J}(t) &=&  {Y_{n-1}+m \choose m-1} - {Y_n+m-1 \choose m-1} + {Y_1+m \choose m-1}  - {Y_{n-1}+m \choose m-1} \\
 & & + (n-2) {Y_1+m-1\choose m-1}- \sum_{j=2}^{n-1} {Y_j+m-1 \choose m-1} + (n-1){X_n+m-1 \choose m-1}\\
 & & -(n-1) {Y_1+m-1 \choose m-1} + {X_n+m-1 \choose m-1} - {Y_1+m \choose m-1} + {X_1+m-1 \choose m-1}\\
 & & - {X_2+m-2 \choose m-1} + {X_2+m-2 \choose m-1} - {X_n+m-2 \choose m-1} - {X_n+m-2 \choose m-2}\\
 & & + \sum_{j=2}^{n-1} {X_j+m-1 \choose m-1} - (n-2) {X_n+m-1 \choose m-1} - {X_n+m-1 \choose m-2} + {t-n\alpha+m-1 \choose m-1}\\
 &=& - \sum_{j=2}^n {Y_j+m-1 \choose m-1} + {Y_1+m \choose m-1} - {Y_1+m-1 \choose m-1} - {Y_1+m \choose m-1}\\
 & & + \sum_{j=1}^{n-1} {X_j+m-1 \choose m-1} + {X_n+m-1 \choose m-1} + {X_n+m-1 \choose m-1} - {X_n+m-2 \choose m-1} \\
 & & - {X_n+m-2 \choose m-2}+ {t-n\alpha+m-1 \choose m-1}\\
 &=&-\sum_{j=1}^n {Y_j+m-1 \choose m-1} + \sum_{j=1}^{n-1} {X_j+m-1 \choose m-1} +2{X_n+m-1 \choose m-1} - {X_n+m-1 \choose m-1}\\
 & & + {t-n\alpha+m-1 \choose m-1}\\
 &=& \sum_{j=1}^n {X_j+m-1 \choose m-1} - \sum_{j=1}^{n} {Y_j +m-1 \choose m-1} + {t-n\alpha+m-1 \choose m-1}.\\
 \end{eqnarray*}

\subsubsection{The Hilbert function of $I^n$}
\label{sec:closesmallnrewrite}

By Proposition \ref{prop:HilbofIn}, the Hilbert function of $I^n$ is equal to 

\begin{eqnarray*}
H_{I^n}(t)&=& \sum_{j=1}^n \Bigg( {t-X_j +(m\!\!-\!\!1) \choose (m\!\!-\!\!1)} - {t-Y_j + (m\!\!-\!\!1) \choose (m\!\!-\!\!1)}\Bigg) +{t-n\alpha+m-1 \choose m-1}.
\end{eqnarray*}

Since $H_J(t) = H_{I^n}(t)$, we conclude by Lemma \ref{lem:determineinvariants} that $J=\text{gin}(I^n)$.  That is, the $\lambda_i$ produced by Algorithm \ref{alg:closesmalln} are the invariants of  $\text{gin}(I^n)$ when $\frac{3}{2}\alpha > \beta >\alpha$ and $2 \leq n < \lceil \frac{\alpha}{\beta-\alpha} \rceil +1$.


\subsection{The Case $\alpha=\beta$, $n \geq 1$.}

\subsubsection{Formulas for the proposed invariants}
\label{sec:equalformulas}

In this section we write a closed-form expression for the invariants produced by Algorithm \ref{alg:equal}.  See Section \ref{apsec:equalformulas} of Appendix \ref{ap:formulas} for details of how this follows from the algorithm.  For $v = qn+j$ where  $q=0, \dots, \alpha-1$ and $j=0, \dots, n-1$,
$$
\lambda_v = (n+1)\alpha-1-q(n+1)-j.
$$

\subsubsection{The Hilbert function of $J$}
Consider $J=(x^k, x^{k-1}y^{\lambda_{k-1}}, \dots, xy^{\lambda_1},y^{\lambda_0})$ where the $\lambda_i$ are given by the formula from Section \ref{sec:equalformulas} and are the invariants produced by Algorithm \ref{alg:equal}.  By Proposition \ref{prop:HilbofJ}, the Hilbert function of $J$ is

\begin{eqnarray*}
H_J(t) &=&  \sum_{i=0}^{n\alpha-1} {t-\lambda_{qn+j}-(qn+j)+(m\!\!-\!\!2) \choose (m\!\!-\!\!2)} + {t-n\alpha+(m\!\!-\!\!1) \choose (m\!\!-\!\!1)}\\
&=& \sum_{q=0}^{\alpha-1} \sum_{j=0}^{n-1} {t-(n+1)\alpha-1-q+m-2 \choose m-2}+ {t-n\alpha+(m\!\!-\!\!1) \choose (m\!\!-\!\!1)}\\
&=& n\Bigg[  {t-(n+1)\alpha+m-1+\alpha \choose m-1} - {t-(n+1)\alpha+m-1 \choose m-1}  \Bigg]+ {t-n\alpha+(m\!\!-\!\!1) \choose (m\!\!-\!\!1)} \\
&=& n\Bigg[  {t-n\alpha+m-1 \choose m-1} - {t-(n+1)\alpha+m-1 \choose m-1}  \Bigg]+ {t-n\alpha+(m\!\!-\!\!1) \choose (m\!\!-\!\!1)}.\\
\end{eqnarray*}

\subsubsection{Rewriting the Hilbert function of $I^n$} 
Under the assumption that $\alpha = \beta$, the Hilbert function of $I^n$ from Proposition \ref{prop:HilbofIn} is
\begin{eqnarray*}
H_{I^n}(t) &=& \sum_{j=1}^n \Bigg[ {t-\alpha(n-j)-\alpha j +m-1 \choose m-1} - {t-\alpha j -\alpha n -\alpha +j\alpha+m-1 \choose m-1} \Bigg]\\
&& + {t-n\alpha+(m\!\!-\!\!1) \choose (m\!\!-\!\!1)}\\
&=& n\Bigg[ {t-n\alpha+m-1 \choose m-1} - {t-\alpha(n+1) +m-1 \choose m-1} \Bigg] + {t-n\alpha+(m\!\!-\!\!1) \choose (m\!\!-\!\!1)}.
\end{eqnarray*}

Since $H_J(t) = H_{I^n}(t)$, Lemma \ref{lem:determineinvariants} implies that $J=\text{gin}(I^n)$.
Thus, the $\lambda_i$ produced by Algorithm \ref{alg:equal} are the invariants of $\text{gin}(I^n)$ when $\alpha = \beta$.

\appendix
\section{Examples}
\label{ap:examples}

This section contains examples of the sequence of numbers $\{ \lambda_i\}$ produced by each algorithm in Section \ref{sec:proposedinvariants}.  We also record the sequence of gaps $\{ g_i \}$ where $g_i = \lambda_{i-1}-\lambda_i$.  It is in these gap sequences that the patterns can be most clearly seen and the spacing used when listing the gap sequences is meant to highlight these patterns.  The figures following the example further illustrate the patterns.

\subsection{$\beta \geq 2\alpha-1$, $n \geq 1$.}  

\begin{itemize}
\item Let $(\alpha, \beta) = (4,12)$, $n=3$.
In this case $\beta -2\alpha+2 = 12-2(4)+2 = 6$.  We have the following sequence $\{\lambda_i \}$ of invariants:
$$\lambda_0 = 39,37,35, 33, 27, 25, 23, 21, 15, 13, 11, 9 = \lambda_{k-1} = \lambda_{11}$$   
The sequence of gaps $\{g_i \}$ is
$$2,2,2,6, \hspace{0.2in} 2,2,2,6, \hspace{0.2in} 2,2,2$$
See Figure \ref{fig:FigureFar} for an illustration of the patterns found in the gap sequence when $(\alpha,\beta) = (4,12)$.

\item Let $(\alpha, \beta) = (4,9)$, $n=4$.  
In this case $\beta -2\alpha+2 = 9-2(4)+2 = 3$.  We have the following sequence $\{\lambda_i \}$ of invariants:
$$\lambda_0 = 39,37,35,33,30,28,26,24,21,19,17,15,12,10,8,6 = \lambda_{k-1} = \lambda_{15}$$   
The sequence of gaps $\{g_i \}$ is
$$2,2,2,3, \hspace{0.2in} 2,2,2,3, \hspace{0.2in} 2,2,2,3, \hspace{0.2in} 2,2,2$$
\end{itemize}

\vspace{0.1in}

\subsection{$2\alpha > \beta > \frac{3}{2} \alpha$, $n \geq 2$.}

\begin{itemize}
\item Let $(\alpha, \beta)=(6, 10)$, $n=5$.

In this case $r = 2(6)-10 = 2$ and $l = 10-6 = 4$.  We have the following sequence $\{\lambda_i \}$ of invariants: 
$$\lambda_0 = 55,53,51,49,47,46,44,43,41,39,37,36,34,33,31, $$
$$29,27,26,24,23,21,19,17,16,14,13,11,9,7,5=\lambda_{k-1} = \lambda_{29}.$$

The sequence of gaps $\{g_i \}$ is 
$$2,2,2,2, \hspace{0.2in}1,2,1, \hspace{0.1in} 2,2,2, \hspace{0.2in}1,2,1, \hspace{0.1in} 2,2,2, \hspace{0.2in} 1,2,1, \hspace{0.1in} 2,2,2, \hspace{0.2in} 1,2,1, \hspace{0.2in}2,2,2,2$$
See Figure \ref{fig:FigureMid} for an illustration of the patterns found in the gap sequence when $(\alpha,\beta) = (6,10)$.

\item Let $(\alpha, \beta)=(7, 12)$, $n=4$.

In this case $r = 2(7)-12 = 2$ and $l = 12-7 = 5$.  We have the following sequence $\{\lambda_i \}$ of invariants: 
$$\lambda_0 = 54, 52, 50, 48, 46, 44, 43, 41, 40 , 38, 36, 34, 32, 31, 29, 28, 26, 24, 22, 20, 19, $$
$$17, 16, 14, 12, 10, 8, 6=\lambda_{k-1} = \lambda_{27}.$$

The sequence of gaps $\{g_i \}$ is 
$$2,2,2,2,2, \hspace{0.2in}1,2, \hspace{0.2in}1,2,2,2,2, \hspace{0.2in}1,2,\hspace{0.2in}1,2,2,2,2, \hspace{0.2in} 1,2, \hspace{0.2in}1,\hspace{0.2in} 2,2,2,2,2$$
\end{itemize}

\vspace{0.1in}

\subsection{$\frac{3}{2} \alpha > \beta > \alpha$, $l | \alpha$, $n \geq \frac{\alpha}{l} +1$.}

\begin{itemize}

\item Let $(\alpha, \beta)=(12, 15)$, $n=5$.

In this case $l = 15-12 = 3$ and $c = 12/3=4$.  We have the following sequence  $\{\lambda_i \}$ of invariants: 
$$\lambda_0 =86,84,82,80,79,77,76,74,73,71,70,69,67,66,65,63,62,61,59,58,57,$$
$$56,54,53,52,51,49,48,47,46,44,43,42,41,39,38,37,36,34,33,32,$$
$$31,29,28,27,25,24,23,21,20,19,17,16,14,13,11,10,8,6,4=\lambda_{k-1}=\lambda_{59}.$$

The sequence of gaps between the $\lambda_i$s is 
\begin{eqnarray*}
&& 2,2,2, \hspace{0.2in}1,2, \hspace{0.1in} 1,2, \hspace{0.1in}1,2,  \hspace{0.2in}1,1,2,  \hspace{0.1in} 1,1,2, \hspace{0.1in}1,1,2,  \hspace{0.2in}1,1,1,2, \hspace{0.1in}1,1,1,2, \hspace{0.1in}1,1,1,2, \\
&& \hspace{0.1in}1,1,1,2 \hspace{0.1in}1,1,1,2, \hspace{0.1in}1,1,1,2, \hspace{0.2in} 1,1, \hspace{0.1in} 2,1,1, \hspace{0.1in} 2,1,1, \hspace{0.2in} 2,1, \hspace{0.1in} 2,1, \hspace{0.1in} 2,1, \hspace{0.2in} 2,2,2
\end{eqnarray*}
See Figure \ref{fig:FigureCloseDivides} for an illustration of the patterns found in the gap sequence when $(\alpha,\beta) = (12,15)$.

\item Let $(\alpha, \beta)=(9, 12)$, $n=4$.

In this case $l = 12-9 = 3$ and $c = 9/3=3$.  We have the following sequence $\{\lambda_i \}$ of invariants: 
$$\lambda_0 =56,54,52,50,49,47,46,44,43,41,40,39,37,36,35,33,32,31,29,$$
$$28,27,25,24,23,21,20,19,17,16,14,13,11,10,8,6,4=\lambda_{k-1} = \lambda_{35}.$$

The sequence of gaps $\{g_i \}$ is 
$$ 2,2,2, \hspace{0.2in}1,2, \hspace{0.1in} 1,2, \hspace{0.1in}1,2,  \hspace{0.2in}1,1,2,  \hspace{0.1in} 1,1,2, \hspace{0.1in}1,1,2, \hspace{0.1in}1,1,2,  \hspace{0.1in}1,1,2, \hspace{0.1in}1,1, $$
$$\hspace{0.2in}2, \hspace{0.1in} 1,\hspace{0.1in}2,1,\hspace{0.1in}2,1,\hspace{0.2in}2,2,2 $$

\end{itemize}


\vspace{0.1in}

\subsection{$\frac{3}{2} \alpha > \beta > \alpha$, $l \nmid \alpha$, $n \geq \lceil \alpha/l \rceil+1$.}

\begin{itemize}

\item Let $(\alpha, \beta)=(10,14)$, $n=4$.  

In this case $l = 14-10 = 4$, $c = \lceil 10/4 \rceil = 3$, and $d = 2$.  We have the following sequence $\{\lambda_i \}$ of invariants: 
$$\lambda_0 =65,63,31,59,57,56,54,53,51,50,48,47,45,44,43,41,40,39,37,$$
$$36,34,33,31,30,29,27,26,25,23,22,20,19,17,16,14,13,11,9,7,5=\lambda_{k-1} = \lambda_{39}.$$

The sequence of gaps $\{g_i \}$ in this example is 
$$2,2,2,2, \hspace{0.2in}1,2, \hspace{0.1in} 1,2, \hspace{0.1in} 1,2, \hspace{0.1in} 1,2, \hspace{0.2in} 1,1,2, \hspace{0.1in} 1,1,2, $$
$$\hspace{0.2in} 1,2, \hspace{0.1in} 1,2, \hspace{0.2in} 1,1,2, \hspace{0.1in} 1,1,2, \hspace{0.2in} 1, \hspace{0.1in} 2,1, \hspace{0.1in} 2,1, \hspace{0.1in} 2,1, \hspace{0.2in} 2,2,2,2$$
See Figure \ref{fig:FigureCloseDoesNotDivide} for an illustration of the patterns found in the gap sequence when $(\alpha,\beta) = (10,14)$.

\item Let $(\alpha, \beta)=(7,9)$, $n=6$.

In this case $l = 12-9 = 3$, $c = \lceil 9/2 \rceil = 5$, and $d=2$.  We have the following sequence $\{\lambda_i \}$ of invariants: 
$$\lambda_0 =60,58,56,55,53,52,50,49,48,46,45,44,42,41,40,39,37,36,35,$$
$$33,32,31,30,28,27,26,24,23,22,21,19,18,17,15,14,13,11,10,8,7,5,3=\lambda_{k-1}= \lambda_{41}.$$

The sequence of gaps $\{g_i \}$ in this example is 
$$2,2, \hspace{0.2in}1,2, \hspace{0.1in} 1,2, \hspace{0.2in} 1,1,2, \hspace{0.1in} 1,1,2, \hspace{0.2in} 1,1,1,2, \hspace{0.1in}1,1,2\hspace{0.2in} 1,1,1,2,\hspace{0.1in}1,1,2,\hspace{0.2in}1,1,1,\hspace{0.2in} 2, \hspace{0.2in} $$
$$1,1, \hspace{0.1in}2,1,1,\hspace{0.2in}2,1,\hspace{0.1in}2,1,\hspace{0.1in}2,2 $$
\end{itemize}


\vspace{0.1in}

\subsection{$\frac{3}{2} \alpha > \beta > \alpha$, $n < \lceil \alpha/l \rceil+1$.}

\begin{itemize}

\item Let $(\alpha, \beta)=(6,8)$, $n=3$.  

In this case $l = 8-6 = 2$.  We have the following sequence $\{\lambda_i \}$ of invariants: 
$$\lambda_0 =29,27,25,24,22,21,19,18,17,15,14,13,11,10,8,7,5,3=\lambda_{k-1} = \lambda_{23}.$$

The sequence of gaps $\{g_i \}$ in this example is 
$$2,2, \hspace{0.2in}1,2, \hspace{0.1in} 1,2,  \hspace{0.2in} 1,1,2, \hspace{0.1in} 1,1,2, \hspace{0.1in}  \hspace{0.15in} 1, \hspace{0.1in} 2,1,  \hspace{0.2in} 2,2$$
See Figure \ref{fig:FigureCloseSmalln} for an illustration of the patterns found in the gap sequence when $(\alpha,\beta) = (6,8)$ and $n=3$.

\item Let $(\alpha, \beta)=(7,10)$, $n=2$.

In this case $l = 10-7 = 3$.  We have the following sequence  $\{\lambda_i \}$ of invariants: 
$$\lambda_0 =26,24,22,20,19,17,16,14,13,11,10,8,6,4=\lambda_{k-1}= \lambda_{13}.$$

The sequence of gaps $\{g_i \}$  in this example is 
$$2,2,2, \hspace{0.2in}1,2, \hspace{0.1in} 1,2, \hspace{0.1in} 1,2, \hspace{0.1in} 1, \hspace{0.1in} 2, \hspace{0.1in} 2,2 $$
\end{itemize}


\vspace{0.1in}

\subsection{$\alpha=\beta$, $n \geq 1$.}

\begin{itemize}

\item Let $(\alpha, \beta) = (3,3)$, $n=5$.  We have the following sequence $\{\lambda_i \}$ of invariants:
$$\lambda_0 = 17, 16, 15, 14, 13, 11, 10, 9, 8, 7, 5, 4, 3, 2, 1 = \lambda_{k-1} = \lambda_{15}.$$
The sequence of gaps $\{g_i \}$ is

$$1,1,1,1,2, \hspace{0.2in}1,1,1,1,2, \hspace{0.1in} 1,1,1,1 .$$
See Figure \ref{fig:FigureEqual} for an illustration of the patterns found in the gap sequence when $(\alpha, \beta) = (3,3)$ and $n=5$.

\item Let $(\alpha,\beta) = (4,4)$, $n=2$. 
 We have the following sequence $\{\lambda_i \}$ of invariants:
 $$\lambda_0 = 11, 10, 8, 7, 5, 4, 2, 1 = \lambda_{k-1} = \lambda_{7}.$$
  The sequence of gaps $\{g_i \}$ is $$1,2, \hspace{0.2in}1,2, \hspace{0.1in} 1,2, \hspace{0.2in} 1$$
 
 \end{itemize}

\newpage

\begin{landscape}

\begin{figure}
\centering
\includegraphics[height=4cm]{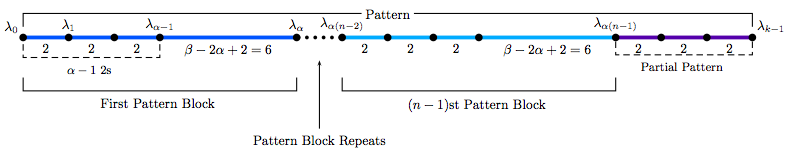}
\caption{Example of the output of Algorithm \ref{alg:far} ($\beta \geq 2\alpha-1$) for case $\alpha=4$ and $\beta=12$. Note that $l = 8$.}
\label{fig:FigureFar}
\end{figure}

\vspace{0.2in}

\begin{figure}
\centering
\includegraphics[height=4.5cm]{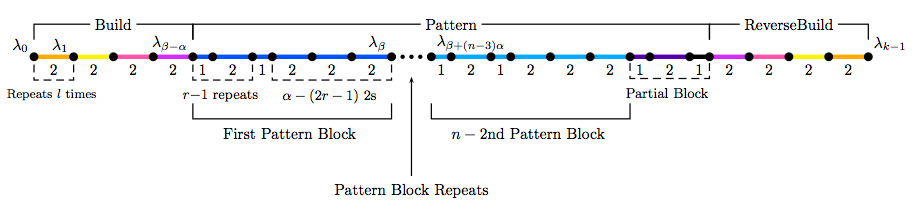}
\caption{Example of the output of Algorithm \ref{alg:mid} ($2\alpha -1 > \beta \geq \frac{3}{2}\alpha$) for case $\alpha=6$ and $\beta=10$. Note that $l=4$ and $r=2\alpha-\beta=2$.}
\label{fig:FigureMid}
\end{figure}

\begin{figure}
\centering
\includegraphics[height=8cm]{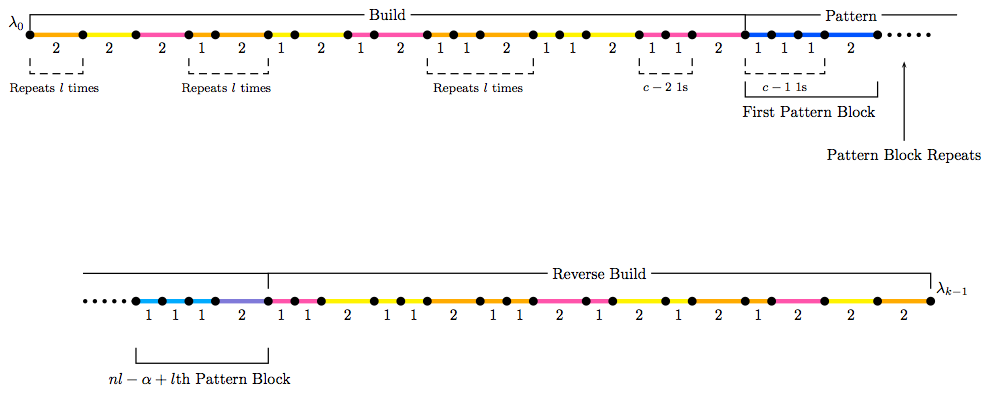}
\caption{Example of the output of Algorithm \ref{alg:closedivides} ($\frac{3}{2}\alpha > \beta >\alpha$, $(\beta-\alpha) | \alpha$, and $n \geq \frac{\alpha}{\beta-\alpha} +1$) for case $\alpha=12$ and $\beta=15$. Note that $l=3$ and $c = \alpha/l=4$.}
\label{fig:FigureCloseDivides}
\end{figure}

\begin{figure}
\centering
\includegraphics[height=7.5cm]{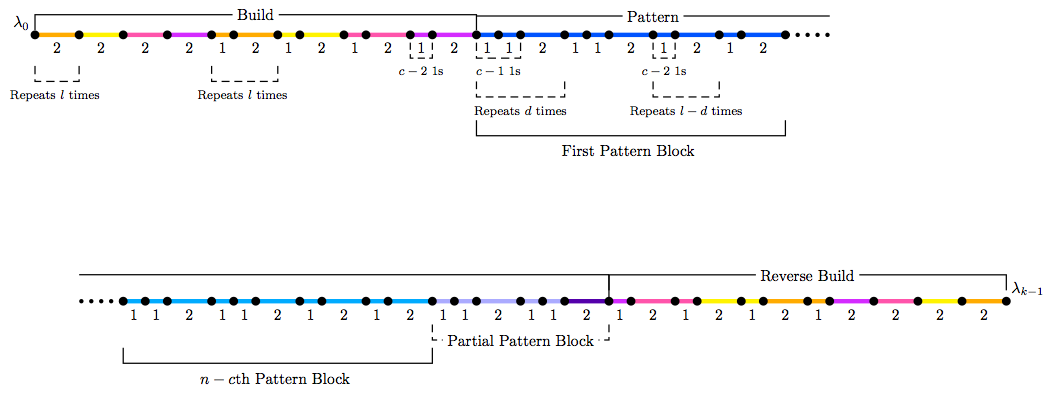}
\caption{Example of the output of Algorithm \ref{alg:closedoesnotdivide} ($\frac{3}{2}\alpha > \beta >\alpha$, $(\beta-\alpha) \nmid \alpha$, and $n \geq \lceil \frac{\alpha}{\beta-\alpha} \rceil+1$) for case $\alpha=10$ and $\beta=14$. Note that $l=4$, $c=3$, and $d = \alpha \mod l = 2$.}
\label{fig:FigureCloseDoesNotDivide}
\end{figure}

\begin{figure}
\centering
\includegraphics[height=3cm]{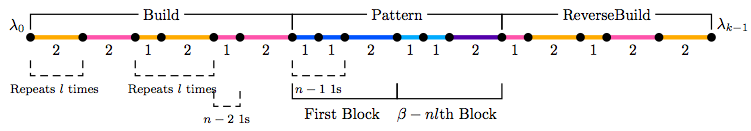}
\caption{Example of the output of Algorithm \ref{alg:closesmalln} ($\frac{3}{2}\alpha > \beta >\alpha$ and $2 \leq n < \lceil \frac{\alpha}{\beta-\alpha} \rceil +1$) for case $\alpha=6$, $\beta=8$, and $n=3$. Note that $l=2$.}
\label{fig:FigureCloseSmalln}
\end{figure}

\vspace{0.2in}

\begin{figure}
\centering
\includegraphics[height=2.8cm]{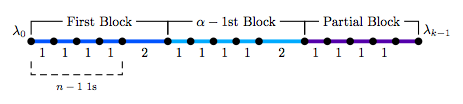}
\caption{Example of the output of Algorithm \ref{alg:equal} ($\alpha=\beta$) for case $\alpha=\beta=3$ and $n=5$.}
\label{fig:FigureEqual}
\end{figure}

\end{landscape}
\section{Details on Formulas for the Invariants}
\label{ap:formulas}

In this section we fill in details of how the formulas for the proposed invariants follow from the algorithms in Section~\ref{sec:proposedinvariants}.  Recall from Section~\ref{sec:proposedinvariants} that each of the algorithms may be divided into three phases:  the Build, the Reverse Build, and the Pattern.  The sequence of gaps $g_i = \lambda_{i-1}-\lambda_i$ arising from the Reverse Build in a particular case is the reverse of the sequence of gaps arising from the Build.  Therefore, we can use closed form expressions for invariants arising from the Build to write expressions for invariants arising from the Reverse Build.  Further, the Pattern consists of repeats of some Pattern Blocks so it will be convenient to write indices within the Pattern in a way that reflects both the number of Pattern Blocks which have passed and the position in the current Pattern Block. Referring to the examples and pictures in Appendix \ref{ap:examples} may help to clarify the work in this section.

\subsection{$\beta \geq 2\alpha-1$, $n \geq 1$}
\label{apsec:farformulas}
Note that each repeat of the Pattern Block \textbf{BlockFar}$(i, \lambda_{i-1},\alpha,\beta)$ produces $\alpha$ invariants.  Then it makes sense to write each index $v$ in the unique form $v=j\alpha+s$ where $0 \leq s < \alpha$ so that $j$ denotes the number of repeats of \textbf{BlockFar}$(i, \lambda_{i-1},\alpha,\beta)$ that have preceded $v$ and $s$ denotes the position of $v$ in \textbf{BlockFar}$(i, \lambda_{i-1},\alpha,\beta)$. Then
\begin{eqnarray*}
\lambda_v &=& \lambda_0 - (\beta-2\alpha+2)j-2(\alpha-1)j-2s\\
&=& (n-j)\beta+\alpha-1-2s
\end{eqnarray*}
where $s = 0, \alpha-1$ and $j=0, \dots, n-1$.

\subsection{$2\alpha -1> \beta \geq \frac{3}{2} \alpha$}
\label{apsec:midformulas}

Throughout this section, $r:=2\alpha-\beta$.

\noindent \textit{Formula for $\lambda_i$ in the Build.}
For $v = 0, \dots, l$, $\sum_{i=1}^v g_i = 2v$ so 
$$\lambda_v = \lambda_0-2v = \alpha(n+1) + l\cdot n -1-2v.$$

\noindent \textit{Formulas for $\lambda_i$ in the Reverse Build}

We can confirm that Algorithm~\ref{alg:mid} produces $\lambda_{k-1} = \beta-\alpha+1$ by computing the sum of the gaps $g_i$ as follows:

\begin{eqnarray*}
\lambda_0 - \lambda_{k-1} &=& \sum_{i=1}^k g_i\\
&=& 2l+[(m-1)\cdot 2 + m+2(\alpha-(2m-1))] +2(m-1) +m +2l\\
&=& \beta(n-1)+2\alpha-2\\
&=& \lambda_0 -(\beta-\alpha+1)
\end{eqnarray*}

Thus, for $v =k-l-1, \dots, k-1$, 
$$\lambda_v = \lambda_{k-1} + \sum_{i=v+1}^{k-1} g_i = \lambda_{k-1} + 2(k-v-1)$$
or, for $i=1, \dots, l+1$,
$$\lambda_{k-i} = (l+1)+2(k-(k-i)-1) = l+1+2(i-1) = l+2i-1.$$

\noindent \textit{Formulas for $\lambda_i$ in the Pattern.}

\vspace{0.1in}
\noindent Note that the total number of invariants produced by one complete repeat of \textbf{BlockMid}$(i, \lambda_{i-1},r,\alpha)$ is $2(r-1)+\alpha-(2r-1)+1 = \alpha$.  Further, $\sum g_i$ where the sum is over one repeat of \textbf{BlockMid}$(i, \lambda_{i-1},r,\alpha)$ is 
$$\alpha + \text{\# of 2s in difference sequence} = \alpha + [r-1 + \alpha-(2r-1)] = 2\alpha-r.$$
Note that the indices $l+1, l+2, \dots, n\alpha-l-2$ are in the Pattern (repeats of \textbf{BlockMid}$(i, \lambda_{i-1},r,\alpha)$); we will divide these indices as follows to reflect different parts of the Pattern.

\begin{eqnarray*}
\text{indices } i & & \text{gaps } g_i\\
l+1, \dots, l+2r-2 &\rightarrow & 12 \cdots 12 \\
l+2r-1, \dots , l+\alpha-1 &\rightarrow& 12\cdots2\\
l+\alpha &\rightarrow& 2\\
\vdots & & \vdots\\
l+(n-3)\alpha+1, \dots, l+(n-3)\alpha+2r-2 & \rightarrow &  12 \cdots 12 \\
l+(n-3)\alpha+2r-1, \dots l+(n-2)\alpha-1 &\rightarrow& 12\cdots2\\
l+(n-2)\alpha  &\rightarrow& 2\\
l+(n-2)\alpha+1, \dots, l+(n-2)\alpha+2(r-1) = n\alpha-l-2 & \rightarrow &  12 \cdots 12 \\
\end{eqnarray*}
   
 In the following, $j$ represents how many full repeats of \textbf{BlockMid} have come before $v$ and $y$ represents the position in the current repeat of \textbf{BlockMid}. 

\begin{itemize}
\item[\ding{202}] $ v= l+j\alpha+y$ where $j=0, \dots, (n-2)$ and $y=1, \dots, 2r-2$.  Then 
$$\lambda_v = \lambda_0 - \bigg[ 2l + (2\alpha-m)j + y+ \Big \lfloor \frac{y}{2} \Big \rfloor \bigg].$$
\item[\ding{203}] $v = l+j\alpha+y$ where $j=0, \dots, (n-3)$ and $y=2r-1, \dots ,\alpha-1$.  Then
$$ \lambda_0 - [2l+(2\alpha-r)j+2y-r]$$
\item[\ding{204}] $v = l+j\alpha$ where $j=1, \dots, n-2$.  Then 
$$\lambda_v = \lambda_0 - [j(2\alpha-r) +2l]$$
\end{itemize}

To get rid of the floor function in the first expression, we break this part into two.  For $p=1, \dots, r-1$, when $y=2p$, $\lfloor y/2\rfloor = p$ and when $y=2p-1$, $\lfloor y/2 \rfloor = p-1$.  We now have the following expressions \textit{replacing} \ding{202}:

\vspace{0.2in}

\begin{itemize}
\item[\ding{205}] $ v= l+j\alpha+2p$ where $j=0, \dots, n-2$ and $p = 1, \dots, r-1 = \alpha-l-1$.  Then
$$\lambda_v = \lambda_0 - [2l+(\alpha+l) j + 2p + p].$$
\item[\ding{206}] $v = l+j\alpha+2p-1$ where $j - 0, \dots, (n-2)$ and $p = 1,\dots, r-1$.  Then 
$$\lambda_v = \lambda_0 - [2l+(\alpha+l)j+2p-2 + p]$$
\end{itemize}

Thus, the expressions denoted \ding{203}, \ding{204}, \ding{205}, and \ding{206} taken together give the formulas for all of the invariants produced by the Pattern phase of Algorithm \ref{alg:mid}.

\subsection{$\frac{3}{2} \alpha > \beta > \alpha$, $l | \alpha$, $n \geq \frac{\alpha}{l} +1$}
\label{apsec:dividesformulas}

Throughout this section $c:=\frac{\alpha}{l}$, which is an integer by assumption.

\textit{Formulas for $\lambda_i$ in the Build.}

Note that the Build in Algorithm \ref{alg:closedivides} spans the indices $1, 2, \dots, l(1+ \cdots +(c-1))$.  We divide these indices to reflect different parts of the Build as follows.

\begin{eqnarray*}
\text{indices } i & & \text{gaps } g_i\\
1, \dots, l &\rightarrow & 2\cdots 2  \\
l+1, \dots, 2l+l &\rightarrow & 12\cdots12  \\
2l+l+1, \dots, 3l+2l+l &\rightarrow & 112 \cdots 112  \\
\vdots & & \vdots  \\
l(1+\cdots+(c-2))+1, \dots, l(1+\cdots+(c-1)) &\rightarrow& 1\cdots12 \vspace{0.1in} \cdots \vspace{0.1in} 1\cdots 12\\
\end{eqnarray*}

This suggests that we write an index $v$ between $l+1$ and $l(1+\cdots+(c-1))$ as $v=l(1+ \cdots + q) + p$ where $q=1, \dots, (c-2)$and $p=1, \dots, (q+1)l$.  We have the following cases:
\begin{enumerate}
\item If $p=(q+1)j$ for $j=1, \dots, l$ then 
$$\sum_{i=1}^v g_i = (2+\cdots+(q+1))l + (q+1)j+ \Big \lfloor \frac{(q+1)j}{q+1}\Big \rfloor = (2+\cdots (q+1))l+(q+1)j+j.$$
\item If $p = (q+1)j-x$ for $j=1, \dots, l$ and $x = 1,\dots, q$ then 
$$\sum_{i=1}^v g_i  = (2+\cdots+(q+1))l+(q+1)j-x+ \Big \lfloor \frac{(q+1)j-x}{q+1} \Big \rfloor = (2+ \cdots(q+1))l+(q+2)j-x-1.$$
\end{enumerate}

Since $\lambda_v = \lambda_0 - \sum_{i=1}^v g_i$,
\begin{itemize}
\item[\ding{202}]  For $v=0, \dots, l$, 
$$\lambda_v = \lambda_0-2v.$$
\item[\ding{203}] For $v=l(1+\cdots+q) +(q+1)j$ where $q=1, \dots, (c-2)$, $j=1, \dots, l$, 
$$\lambda_v = \lambda_0-[(2+\cdots(q+1))l+(q+1)j+j].$$
\item[\ding{204}] For $v=l(1+\cdots q) +(q+1)j -x$ where $q=1, \dots, (c-2)$, $j=1, \dots, l$, $x = 1,\dots, q$, 
$$\lambda_v=\lambda_0-[(2+\cdots+(q+1))l+(q+1)j-x+j-1].$$
\end{itemize}

\noindent \textit{Formula for $\lambda_i$ in the Reverse Build.}

We claim that the value of $\lambda_{k-1}$ produced by Algorithm \ref{alg:closedivides} is equal to $l+1$ as we would expect by Theorem \ref{thm:ginstructure}.  This can be checked as follows:
\begin{eqnarray*}
\lambda_0-\lambda_{k-1} &=& \sum_{i=1}^{k-1} g_i\\
&=& l(2+\cdots + c) +[(c-1)+2][nl-\alpha+l] + l(2+\cdots+c)-2\\
&=& \alpha-l+\alpha n +nl-2\\
&=& \beta(n-1)+2\alpha-2 = \lambda_0-(l+1)
\end{eqnarray*}
\vspace{0.1in}

The sequence of gaps in the Reverse Build is almost the opposite of the sequence of gaps in the Build; the exception is that the final part of the Build which does not appear in the Reverse Build.\footnote{This is why the \textbf{ReverseBuildPartial} subroutine is necessary}.  Thus, we can use the formulas for $\sum g_i$ derived above to find formulas for $\lambda_v$ where $v$ is in the Build.  In particular, 
$$\lambda_v = \lambda_{k-1-T} = \lambda_{k-1}+\sum_{i=0}^T g_i$$
where $v$ ranges over the index set of the Build.  We take $v = \lambda_{k-1-T}$ where $T = 0, \dots, l(1+\cdots (c-2))+(c-1)l-1$.\footnote{Note that for the Build we considered indices $t=0, \dots, l(1+\cdots + (c-2))+(c-1)l$).}  The final formulas are listed in Section \ref{sec:dividesformulas} of the main text.

\noindent \textit{Formulas for $\lambda_i$ in the Pattern.}

We introduce some additional notation.  Let $E$ be the final index of the Build, $E:= l(1+\cdots+(c-1))$ and let $B := \sum_{i=1}^E g_i$ so $B=l(2+\cdots + c).$
Then $v = E+1, E+1, \dots, E+(ln-\alpha+l)c$ are the indices that are part of the Pattern.  Since each repeat of the \textbf{BlockClose}$(i, \lambda_{i-1},c,d,l,\alpha)$ subroutine produces $c$ invariants, we divide these indices as follows to reflect different parts of the Pattern.

\begin{eqnarray*}
\text{indices } i & & \text{gaps } g_i\\
E+1, \dots, E+c-1 &\rightarrow & 1\cdots1 \\
E+c &\rightarrow & 2  \\
E+c+1, \dots,E+ c+(c-1) &\rightarrow & 1 \cdots 1  \\
E+2c &\rightarrow & 2\\
\vdots & & \vdots  \\
E+(ln-\alpha+l-1)c+1, \dots , (ln-\alpha+l-1)c+(c-1) &\rightarrow& 1\cdots 1\\
E+(ln-\alpha+l)c &\rightarrow & 2
\end{eqnarray*}

This suggests the following formulas.  

\begin{itemize}
\item[\ding{202}]  For $v=E+jc+i$ where $j=0, \dots, ln-\alpha+l-1$, $i=1, \dots, c-1$,
$$\lambda_v = \lambda_0-B-[j(c+1)+i]$$

\item[\ding{203}] For $v = E+jc$ where $j=1, \dots, ln-\alpha+l-1$,
$$\lambda_v = \lambda_0-B-[j(c+1)]$$
\end{itemize}


\subsection{$\frac{3}{2} \alpha > \beta > \alpha$, $l  \nmid \alpha$, $n \geq \lceil \frac{\alpha}{l} \rceil+1$}
\label{apsec:notdivideformulas}

Throughout this section, set $c:=\lceil \frac{\alpha}{l} \rceil$ and $d:=\alpha \mod l$.

\noindent \textit{Formulas for $\lambda_i$ in the Build and Reverse Build.}

The Build and the Reverse Build are identical to the $\frac{3}{2}\alpha > \beta > \alpha$, $l | \alpha$, $n \geq \frac{\alpha}{l}$ case.\footnote{This is because the \textbf{Build}, \textbf{ReverseBuild}, and \textbf{ReverseBuildPartial} with the same inputs are found in both Algorithm \ref{alg:closedivides} and Algorithm \ref{alg:closedoesnotdivide}.} Further, the $\lambda_{k-1}$  value produced by Algorithm \ref{alg:closedoesnotdivide} is equal to $l+1$ since 
\begin{eqnarray*}
\lambda_0-\lambda_{k-1}&=& \sum_{i=0}^{k-1} g_i \\
&=& 2l(2+ \cdots + c) + (n-c)[d(c+1)+(l-d)c] -2+d(c+1)\\
&=& \beta(n-1)+2\alpha-2.
\end{eqnarray*}

Therefore, the formulas for $\lambda_i$ in the Build and the Reverse Build are the same as those in Section \ref{apsec:dividesformulas}.

\noindent \textit{Formulas for $\lambda_i$ in the Pattern.}

\noindent As before, we introduce some notation to denote where the Pattern begins.  Let $E$ be the index where the Build ends so that 
$$E:= l(1+\cdots+(c-1)).$$
Let $B = \sum_{i=1}^E g_i$, so that 
$$B:=l(2+\cdots +c).$$

Consider the first Pattern Block; we divide the indexing as follows
\begin{eqnarray*}
\text{indices } i & & \text{gaps } g_i\\
E+1, \dots, E+c-1 &\rightarrow & 1\cdots1 \\
E+c &\rightarrow& 2\\
E+c+1, \dots, E+2c-1 \rightarrow & 1 \cdots 1\\
E+2c &\rightarrow &2\\
\vdots & & \vdots\\
E+(d-1)c+1, \dots, E+(d-1)c+(c-1) &\rightarrow & 1\cdots 1\\
E+dc &\rightarrow & 2\\
E+dc+1, \dots, E+dc+(c-2) &\rightarrow & 1\cdots 1\\
E+dc+(c-1) &\rightarrow & 2\\
E+dc+(c-1)+1, \dots, E+dc+(c-1)+(c-2) &\rightarrow & 1\cdots 1\\
E+dc+2(c-1) &\rightarrow & 2\\
\vdots & & \vdots\\
E+dc+(l-d-1)(c-1)+1, \dots, E+dc+(l-d-1)(c-1)+(c-2)&\rightarrow & 1\cdots 1\\
E+dc+(l-d)(c-1)=E+\alpha &\rightarrow &2\\
\end{eqnarray*}

Note that each Pattern Block has total length $\alpha$ so we divide the remainder of the indexing in the same way.  We then the following formulas.  \ding{202} and \ding{203} describe $v$ such that $g_v$ are the 1s and 2s respectively from $1\cdots1 2$ with $(c-1)$ 1s. \ding{204} and \ding{205} describe $v$ such that $g_v$ are the 1s and 2s from the $1\cdots 12$ with $(c-2)$ 1s.  Note that $p$ signifies the number of full patterns that have preceded the one that $v$ belongs to.

\begin{itemize}
\item[\ding{202}]  For $v = E+p\alpha+jc+i$ where $p=0, \dots, n-c$, $j=0,\dots d-1$, $i=1, \dots, c-1$,
$$\lambda_v=\lambda_0-B-[p(l+\alpha)+j(c+1)+i]$$
\item[\ding{203}] For $v = E+p\alpha+jc$ where $p=0, \dots, n-c$, $j=1, \dots, d$, 
$$\lambda_v = \lambda_0-B-[p(l+\alpha+j(c+1)]$$
\item[\ding{204}] For $v=E+p\alpha+dc+j(c-1)+i$ where $p=0, \dots, n-c-1$, $j=0, \dots, l-d-1$, $i=1, \dots, c-2$
$$\lambda_v=\lambda_0-B-[p(l+\alpha)+d(c+1)+jc+i]$$
\item[\ding{205}] For $v=E+p\alpha+dc+j(c-1)$ where $p=0, \dots, n-c-1$, $j=1, \dots, l-d$
$$\lambda_v = \lambda_0-B-[p(l+\alpha)+d(c+1)+jc]$$
\end{itemize}

\subsection{$\frac{3}{2}\alpha > \beta > \alpha$, $2\leq n < \lceil \frac{\alpha}{l} \rceil +1$}
\label{apsec:closesmallnformulas}

Algorithm \ref{alg:closesmalln} for this case is similar to Algorithm \ref{alg:closedivides} for the case where  $\frac{3}{2} \alpha > \beta > \alpha$, $l | \alpha$, $n \gg 0$ so we will refer to the results in Section \ref{sec:divides} and only mention where replacements are made

\textit{Formulas for $\lambda_i$ the Build.}
The Build is the same as in the $\frac{3}{2} \alpha > \beta > \alpha$, $l | \alpha$, $n \gg 0$ case except with $c$ replaced by $n$.
\begin{itemize}
\item[\ding{202}] For $v=0, \dots, l$, 
$$\lambda_v = \lambda_0-2v.$$
\item[\ding{203}] For $v=l(1+\cdots+q) +(q+1)j$ where $q=1, \dots, (n-2)$, $j=1, \dots, l$, 
$$\lambda_v = \lambda_0-[(2+\cdots+(q+1))l+(q+1)j+j]$$
\item[\ding{204}] For $v=l(1+\cdots + q) +(q+1)j -x$ where $q=1, \dots, (n-2)$, $j=1, \dots, l$, $x = 1,\dots, q$, 
$$\lambda_v=\lambda_0-[(2+\cdots+(q+1))l+(q+1)j-x+j-1]$$
\end{itemize}

\textit{Formulas for $\lambda_i$ in the Reverse Build.}
The Build is the same as in the $\frac{3}{2} \alpha > \beta > \alpha$, $n \gg 0$ case except with $c$ replaced by $n$.  It is easy to check that $\sum_{i=1}^{k-1} g_i = \lambda_0 - (l+1)$ so these formulas are valid when we set $\lambda_{k-1} = l+1$.
\begin{itemize}
\item[\ding{202}]  For $v=(k-1)-j$ where  $j=0, \dots, l$, 
$$\lambda_v = \lambda_{k-1}+2j.$$
\item[\ding{203}] For $v=(k-1)-(l(1+\cdots+q) +(q+1)j)$ where $q=1, \dots, (n-2)$, $j=1, \dots, l$,
$$\lambda_v = \lambda_{k-1}+[(2+\cdots+(q+1))l+(q+1)j+j].$$
However, when $q=n-2$ and $j=l$, $\lambda_v$ is in the Pattern, not in the Reverse Build.
\item[\ding{204}] For $v=(k-1)-(l(1+\cdots + q) +(q+1)j -x)$ where $q=1, \dots, (n-2)$, $j=1, \dots, l$, $x = 1,\dots, q$, 
$$\lambda_v=\lambda_{k-1}+[(2+\cdots+(q+1))l+(q+1)j-x+j-1].$$
However, when $q=n-2$, $j=l$, and $x=1$, $\lambda_v$ is in the Pattern, not the Reverse Build.
\end{itemize}

\textit{Formulas for $\lambda_i$ in the Pattern.}

Let $E$ be the index where the Build ends; that is
$$E:= l(1+\cdots+(n-1)).$$
Let $B = \sum_{i=1}^E g_i$ so 
$$B:=l(2+\cdots + n).$$

The Pattern is similar to the case where $\frac{3}{2} \alpha > \beta > \alpha$, $l |\alpha$, and $n \gg 0$ except that $c$ is replaced by $n$ within each Pattern Block and the Pattern repeats $\beta-nl$  times instead of $nl-\alpha+l$ times.  Thus, we can make appropriate substitutions in the formulas for the invariants in the Pattern from Section \ref{sec:dividesformulas} to determine the formulas here.

\begin{itemize}
\item[\ding{202}]  For $v=E+jn+i$ where $j=0, \dots, \beta-nl-1$, $i=1, \dots, n-1$,
$$\lambda_v = \lambda_0-B-[j(n+1)+i].$$
\item[\ding{203}] For $v = E+jn$ where $j=1, \dots, \beta-nl-1$,
$$\lambda_v = \lambda_0-B-[j(n+1)].$$
\end{itemize}

\subsection{$\alpha=\beta$, $n \geq 1$}
\label{apsec:equalformulas}

The only phase of Algorithm \ref{alg:equal} is the Pattern.  Note that the Pattern Block subroutine \textbf{onestwo}$(n-1, i, \lambda_{i-1})$ produces $n$ invariants.  Thus, it is convenient to write $v = qn+j$ where $0 \leq j \leq n-1$.  Since $\lambda_{qn}-\lambda_{(q+1)n} = n+1$ and $\lambda_{qn+j} -\lambda_{qn+j+1} = 1$ for $j \neq n-1$, we have 
\begin{eqnarray*}
\lambda_v &=& \lambda_0 - q(n+1)-j \\
&=& (n+1)\alpha-1-q(n+1)-j.
\end{eqnarray*}
for $q=0, \dots, \alpha-1$ and $j=0, \dots, n-1$.

\section{Calculation Details}
\label{ap:calculations}

This appendix contains the details of the routine calculations found in Section~\ref{sec:proofofmain}.  In particular, we rewrite $H_{I^n}(t)$ from Proposition \ref{prop:HilbofIn} according to the relations between $\alpha$ and $\beta$ and we simplify partial sums of the form $\sum_i {t-\lambda_i-i+ m-2  \choose  m-2 }$ where $i$ ranges over subsets of the invariants.  These simplified expressions are added together in Section \ref{sec:proofofmain} to give
$$H_J(t)= \sum_{i=0}^{n\alpha-1} {t-\lambda_i-i+ m\!\!-\!\!2  \choose  m\!\!-\!\!2 } + {t-n\alpha+ m\!\!-\!\!1  \choose  m\!\!-\!\!1 }$$  
as in Proposition \ref{prop:HilbofJ}.
The numbering of subsets of the Build, Reverse Build, and Pattern matches that in Appendix~\ref{ap:formulas} and in the main text.  The only combinatorial identities used are equations (1) through (4) from Section \ref{sec:prelim}.

\subsection{$\beta \geq 2\alpha-1$, $n \geq 1$}
All details are contained in the main text.


\subsection{ $2\alpha -1> \beta \geq \frac{3}{2} \alpha$}  
\label{apsec:midcalculations}
We begin by computing $\lambda_v+v$ and then simplify the sums of the associated binomial coefficients.  Set $X_j = t-n\alpha-jl$ and $Y_j = t-(n+1)\alpha-jl.$

\vspace{0.1in}
\noindent \textit{Partial sums from the invariants in the Pattern.}

For $v = 0, \dots, l$,
$$\lambda_v+v = \lambda_0-2v+v = (n+1)\alpha+nl-1-v.$$
Then
\begin{eqnarray*}
& & \sum_{v=0}^l {t-[(n+1)\alpha+nl-1-v]+m\!\!-\!\!2 \choose m\!\!-\!\!2} \\
&=& \sum_{v=0}^l {t-(n+1)\alpha-nl+m\!\!-\!\!1 +v \choose m\!\!-\!\!2} \\
&=& {t-(n+1)\alpha-nl+m\!\!-\!\!1 +l+1 \choose m\!\!-\!\!1} - {t-(n+1)\alpha-nl+m\!\!-\!\!1 \choose m\!\!-\!\!1}\\
&=& {Y_{n-1}+m \choose m\!\!-\!\!1} - {Y_n+m\!\!-\!\!1 \choose m\!\!-\!\!1}.
\end{eqnarray*}

\noindent\noindent \textit{Partial sums from the invariants in the Reverse Build.}
\vspace{0.1in}

For $i=1, \dots, l+1$,
$$\lambda_{k-i} + {k-i} = (l+2i-1) + (k-i) = l+i+n\alpha-1.$$
Then
\begin{eqnarray*}
& & \sum_{i=1}^{l+1} {t-[l+n\alpha+i-1]+m\!\!-\!\!2 \choose m\!\!-\!\!2}\\
&=& \sum_{i=1}^{l+1} {t-l-n\alpha+m\!\!-\!\!1-i \choose m\!\!-\!\!2}\\
&=& {t-l-n\alpha+m\!\!-\!\!1 \choose m\!\!-\!\!1} - {t-l-n\alpha+m\!\!-\!\!1-(l+1) \choose m\!\!-\!\!1}\\
&=& {X_1+m\!\!-\!\!1\choose m\!\!-\!\!1} - {X_2+m\!\!-\!\!2 \choose m\!\!-\!\!1}.
\end{eqnarray*}

\vspace{0.2in}
\noindent \textit{Partial sums from the invariants in the Pattern.}
\vspace{0.1in}

\noindent \ding{203} For $v = l+j\alpha+y$ where $j=0, \dots, (n-3)$ and $y=2r-1, \dots, \alpha-1$,
\begin{eqnarray*}
\lambda_v+v &=& \lambda_0-[2l+(\alpha+l)j+2y-(\alpha-l)]+l+j\alpha+y \\
&=& \lambda_0 -l -jl-y+(\alpha-l)\\
&=& (n+2)\alpha+l(n-j-2)-1-y.
\end{eqnarray*}
Then
\begin{eqnarray*}
& & \sum_{j=0}^{n-3} \sum_{y=2r-1}^{\alpha-1} {t-[(n+2)\alpha+l(n-j-2)-1-y] + m\!\!-\!\!2  \choose  m\!\!-\!\!2 }\\
&=& \sum_{j=0}^{n-3} \sum_{y=2(\alpha-l)-1}^{\alpha-1} {t-(n+2)\alpha-l(n-j-2) + m\!\!-\!\!1+y \choose  m\!\!-\!\!2 }\\
&=& \sum_{j=0}^{n-3} \Bigg[ {t-(n+2)\alpha-l(n\!-\!j\!-\!2)+ m\!\!-\!\!1+\alpha \choose  m\!\!-\!\!1} \\
& & - {t-(n+2)\alpha-l(n\!-\!j\!-\!2)+ m\!\!-\!\!1+2(\alpha-l)-1 \choose 3} \Bigg] \\
&=& \sum_{j=0}^{n-3} \Bigg[ {t-(n+1)\alpha-l(n-j-2)+ m\!\!-\!\!1  \choose  m\!\!-\!\!1} - {t-n\alpha-l(n-j)+ m\!\!-\!\!2  \choose  m\!\!-\!\!1} \Bigg] \\
&=& \sum_{p=3}^n \Bigg[ {t-(n+1)\alpha-l(p-2)+ m\!\!-\!\!1 \choose  m\!\!-\!\!1} - {t-n\alpha-lp+ m\!\!-\!\!2  \choose  m\!\!-\!\!1} \Bigg]\\
&=& \sum_{j=1}^{n-2} {Y_j+ m\!\!-\!\!1 \choose  m\!\!-\!\!1} - \sum_{j=3}^n {X_j + m\!\!-\!\!2  \choose  m\!\!-\!\!1}
\end{eqnarray*}

\noindent \ding{204} For $v = l+j\alpha$ where $j=1, \dots, n-2$,
\begin{eqnarray*}
\lambda_v+v &=& \lambda_0-[j(\alpha+l)+2l]+l+j\alpha\\
&=& \lambda_0-jl-l\\
&=& (n+1)\alpha+l(n-j-1)-1.
\end{eqnarray*}
Then
\begin{eqnarray*}
& & \sum_{j=1}^{n-2} {t-[(n+1)\alpha+l(n-j-1)-1]+ m\!\!-\!\!2  \choose  m\!\!-\!\!2 }\\
&=& \sum_{j=1}^{n-2} {t-(n+1)\alpha-l(n-j-1)+ m\!\!-\!\!1 \choose  m\!\!-\!\!2 }\\
&=&\sum_{p=1}^{n-2} {t-(n+1)\alpha-lp+ m\!\!-\!\!1 \choose  m\!\!-\!\!2 }\\
&=& \sum_{j=1}^{n-2} {Y_j+ m\!\!-\!\!1 \choose  m\!\!-\!\!2 }
\end{eqnarray*}

\noindent \ding{205}   For $ v= l+j\alpha+2p$ where $j=0, \dots, n-2$ and $p = 1, \dots, r-1 = \alpha-l-1$,
\begin{eqnarray*}
\lambda_v +v &=& \lambda_0-[2l+(\alpha+l)j+ 3p] + l + j\alpha+2p \\
&=& \lambda_0 - l-lj-p\\
&=&(n+1)\alpha+(n-j-1)l-1-p.
\end{eqnarray*}
Then
\begin{eqnarray*}
& & \sum_{j=0}^{n-2} \sum_{p=1}^{\alpha-l-1} {t-[(n+1)\alpha+(n-j-1)l-1-p]+ m\!\!-\!\!2  \choose  m\!\!-\!\!2 } \\
&=& \sum_{j=0}^{n-2} \Bigg[ {t-(n+1)\alpha-(n\!-\!j\!-\!1)l+ m\!\!-\!\!1 +\alpha-l \choose  m\!\!-\!\!1 } - {t-(n+1)\alpha-(n\!-\!j\!-\!1)l+ m\!\!-\!\!1 +1 \choose  m\!\!-\!\!1 } \Bigg] \\
&=& \sum_{j=0}^{n-2} \Bigg[ {t-n\alpha-(n\!-\!j)l+ m\!\!-\!\!1  \choose  m\!\!-\!\!1 } - {t-(n+1)\alpha-(n\!-\!j-\!1)l+m \choose  m\!\!-\!\!1 } \Bigg]\\
&=& \sum_{p=2}^n {X_p + m\!\!-\!\!1  \choose  m\!\!-\!\!1 } -\sum_{p'=1}^{n-1} {Y_{p'}+m\choose  m\!\!-\!\!1 }
\end{eqnarray*}

\noindent \ding{206} For $v = l+j\alpha+2p-1$ where $j = 0, \dots, (n-2)$ and $p = 1,\dots, r-1$, 
\begin{eqnarray*}
\lambda_v+v &=& \lambda_0-[2l+(\alpha+l)j+3p-2] + l+j\alpha+2p-1 \\
&=& \lambda_0 - l-lj-p+1\\
&=& (n+1)\alpha+l(n-j-1) -p.
\end{eqnarray*}
Then
\begin{eqnarray*}
& & \sum_{j=0}^{n-2} \sum_{p=1}^{\alpha-l-1} {t-[(n+1)\alpha+l(n-j-1)-p]+ m\!\!-\!\!2  \choose  m\!\!-\!\!2 }\\
&=& \sum_{j=0}^{n-2} \sum_{p=1}^{\alpha-l-1} {t-(n+1)\alpha-l(n-j-1) + m\!\!-\!\!2  +p \choose  m\!\!-\!\!2 }\\
&=& \sum_{j=0}^{n-2} \Bigg[ {t-(n+1)\alpha-l(n\!-\!j\!-\!1) + m\!\!-\!\!2 +\alpha-l \choose  m\!\!-\!\!1 } - {t-(n+1)\alpha-l(n\!-\!j\!-\!1)+ m\!\!-\!\!2+1  \choose  m\!\!-\!\!1 } \Bigg] \\
&=& \sum_{p=2}^n \Bigg[ {t-n\alpha-lp+ m\!\!-\!\!2  \choose  m\!\!-\!\!1 } - {t-(n+1)\alpha-l(p-1)+ m\!\!-\!\!1  \choose  m\!\!-\!\!1 } \Bigg]\\
&=& \sum_{j=2}^n{X_j+ m\!\!-\!\!2  \choose  m\!\!-\!\!1 } - \sum_{j=1}^{n-1} {Y_j+ m\!\!-\!\!1  \choose  m\!\!-\!\!1 }
\end{eqnarray*}

Adding these partial sums together with ${t-n\alpha+m-1 \choose m-1}$ gives the Hilbert funcion sum in Section~\ref{sec:midsum}.


\subsection{$\frac{3}{2} \alpha > \beta > \alpha$, $l | \alpha$, $n \geq \frac{\alpha}{l} +1$}  
\label{apsec:dividescalculations}

Throughout this section, $c:=\frac{\alpha}{l}$; this is an integer by assumption.

\vspace{0.2in}

\noindent \textit{Partial sums from invariants in the Build.}

Recall that $\lambda_0=(n+1)\alpha+ln-1$ and set $X_j = t-n\alpha-jl$ and $Y_j = t-(n+1)\alpha-jl$ as above.

\noindent \ding{202} 

 For $v=0, \dots, l$, 
$$\lambda_v+v = \lambda_0-2v+v = \lambda_0-v = (n+1)\alpha+nl-v-1.$$
Then
\begin{eqnarray*}
& & \sum_{j=0}^l {t-((n+1)\alpha+nl-j-1)+m\!\!-\!\!2 \choose m\!\!-\!\!2} \\
&=& \sum_{j=0}^l { t-(n+1)\alpha-nl+(m\!\!-\!\!1)+j \choose m\!\!-\!\!2 } \\
&=& \Bigg[ {t-(n+1)\alpha-nl+m\!\!-\!\!1 + l +1 \choose m\!\!-\!\!1} - {t-(n+1)\alpha-nl+m\!\!-\!\!1 \choose m\!\!-\!\!1} \Bigg]\\
&=& {Y_{n-1}+m \choose m\!\!-\!\!1} -{Y_n +m\!\!-1 \choose m\!\!-\!\!1}
\end{eqnarray*}

\newpage
 \noindent \ding{203}

For $v=l(1+\cdots+q) +(q+1)j$ where $q=1, \dots, (c-2)$, $j=1, \dots, l$,
\begin{eqnarray*}
\lambda_v+v &=& \lambda_0-(2+\cdots+(q+1))l-(q+1)j-j+l(1+\cdots +q) + (q+1)j \\
&=& (n+1)\alpha+l(n-q)-j-1.
\end{eqnarray*}
Then
\begin{eqnarray*}
& & \sum_{q=1}^{c-2} \sum_{j=1}^l {t-(n+1)\alpha-l(n-q)+j +1 +m\!\!-\!\!2 \choose m\!\!-\!\!2}\\
&=& \sum_{q=1}^{c-2} \Bigg[ {t-(n+1)\alpha-l(n-q)+m\!\!-\!\!1 +l+1 \choose m\!\!-\!\!1}-{t-(n+1)\alpha-l(n-q)+m \choose m\!\!-\!\!1} \Bigg]\\
&=& \sum_{j=n-c+1}^{n-2} {Y_j+m \choose m\!\!-\!\!1} - \sum_{j=n-c+2}^{n-1} {Y_j+m \choose m\!\!-\!\!1}\\
&=& {Y_{n-c+1} + m \choose m\!\!-\!\!1}-{Y_{n-1}+m \choose m\!\!-\!\!1}
\end{eqnarray*}

\noindent \ding{204} 

For $v=l(1+\cdots + q) +(q+1)j -x$ where $q=1, \dots, (c-2)$, $j=1, \dots, l$, $x = 1,\dots, q$,
\begin{eqnarray*}
\lambda_v+v &=& \lambda_0-(2+\cdots+(q+1))l-(q+1)j+x-j+1+l(1+\cdots+q)+(q+1)j-x \\
&=&  (n+1)\alpha+l(n-q)-j.
\end{eqnarray*}
Then
\begin{eqnarray*}
& & \sum_{q=1}^{c-2} q \Bigg[\sum_{j=1}^l {t-(n+1)\alpha-l(n-q)+j+m\!\!-\!\!2 \choose m\!\!-\!\!2}\Bigg]\\
&=& \sum_{q=1}^{c-2} q \Bigg[ {t-(n+1)\alpha-l(n-q)+m\!\!-\!\!2 +l+1 \choose m\!\!-\!\!1} -{t-(n+1)\alpha-l(n-q)+m\!\!-\!\!2+1 \choose m\!\!-\!\!1} \Bigg]\\
&=& \sum_{j=n-c+1}^{n-2} (n-j-1) {t-(n+1)\alpha-lj+m\!\!-\!\!1 \choose m\!\!-\!\!1} - \sum_{j'=n-c+2}^{n-1} (n-j') {t-(n+1)\alpha-lj'+m\!\!-\!\!1 \choose m\!\!-\!\!1}\\
&=& (n-(n-c+1)-1) {Y_{n-c+1}+m\!\!-\!\!1 \choose m\!\!-\!\!1} -(n-(n-1)-1) {Y_{n-1}+m\!\!-\!\!1 \choose m\!\!-\!\!1}\\
&& - \sum_{j=n-c+2}^{n-1} {Y_j+m\!\!-\!\!1 \choose m\!\!-\!\!1}\\
&=&(c-2){Y_{n-c+1}+m\!\!-\!\!1 \choose m\!\!-\!\!1} - \sum_{j=n-c+2}^{n-1} {Y_j+m\!\!-\!\!1 \choose m\!\!-\!\!1}
\end{eqnarray*}

\vspace{0.2in}

\newpage

\noindent \textit{Partial sums from invariants in the Reverse Build.}

Recall that $\lambda_{k-1} = l+1$.

\noindent \ding{202}

For $v=(k-1)-j$ where  $j=0, \dots, l$, $$\lambda_v+v = \lambda_{k-1}+2j+(k-1)-j = l+n\alpha+j.$$
Then
\begin{eqnarray*}
& & \sum_{j=0}^l {t-n\alpha-l-j+m\!\!-\!\!2 \choose m\!\!-\!\!2}\\
&=& {t-n\alpha-l+m\!\!-\!\!2+1 \choose m\!\!-\!\!1} -{t-n\alpha-2l+m\!\!-\!\!2\choose m\!\!-\!\!1}\\
&=& {X_1+ m\!\!-\!\!1 \choose m\!\!-\!\!1} - {X_2+m\!\!-\!\!2 \choose m\!\!-\!\!1}
\end{eqnarray*}

\noindent \ding{203}

For $v=(k-1)-(l(1+\cdots+m) +(q+1)j)$ where $q=1, \dots, (c-2)$, $j=1, \dots, l$, with the exception of $q=c-2$ and $j=l$,
\begin{eqnarray*}
 \lambda_v+v&=& \lambda_{k-1} +(2+\cdots+(q+1))l+(q+1)j+j+(k-1)-(l(1+\cdots +q)+(q+1)j)\\
 &=& (q+1)l+n\alpha+j.
 \end{eqnarray*}
Note that $\lambda_v$ where $q=c-2$ and $j=l$ is not in the Reverse Build. Then the sum $\sum_v {t-v-\lambda_v+m-2 \choose m-2}$ where $v$ ranges over this part of the Reverse Build is:\begin{eqnarray*}
& & \sum_{q=1}^{c-2} \sum_{j=1}^l {t-n\alpha-(q+1)l-j+m\!\!-\!\!2 \choose m\!\!-\!\!2}- {t-(n+1)\alpha+m-2 \choose m-2} \\
&=& \sum_{q=1}^{c-2} \Bigg[ {t-n\alpha-(q+1)l+m\!\!-\!\!2 \choose m\!\!-\!\!1} -{t-n\alpha-(q+1)l+m\!\!-\!\!2-l \choose m\!\!-\!\!1} \Bigg] - {Y_0+m-2 \choose m-2}\\
&=& \sum_{j=2}^{c-1} {X_j+m\!\!-\!\!2 \choose m\!\!-\!\!1} -\sum_{j'=3}^c {X_{j'}+m\!\!-\!\!2\choose m\!\!-\!\!1}-{Y_0+m-2 \choose m-2}\\
&=& {X_2+m\!\!-\!\!2 \choose m\!\!-\!\!1} - {X_c+m\!\!-\!\!2 \choose m\!\!-\!\!1}-{Y_0+m-2 \choose m-2}
\end{eqnarray*}

\noindent \ding{204}

  For $v=(k-1)-(l(1+\cdots + q) +(q+1)j -x)$ where $q=1, \dots, (c-2)$, $j=1, \dots, l$, $x = 1,\dots, q$, with the exception of $q=c-2, j=l$, and $x=1$,
\begin{eqnarray*}
\lambda_v+v &=& \lambda_{k-1} +(2+\cdots +(q+1))l-x+j-1+(k-1)-l(1+\cdots +q)+x\\
&=& n\alpha+(q+1)l+j-1.
\end{eqnarray*}
Note that $\lambda_v$ where $q=c-2$, $j=l$, and $x=1$ is not in the Reverse Build. Then the sum $\sum_v {t-v-\lambda_v+m-2 \choose m-2}$ where $v$ ranges over this part of the Reverse Build is:
\begin{eqnarray*}
& & \sum_{q=1}^{c-2} q \Bigg[ \sum_{j=1}^l {t-n\alpha-(q+1)l-j-1+m\!\!-\!\!2\choose m\!\!-\!\!2}\Bigg] - {t-(n+1)\alpha+m-1 \choose m-2} - {Y_0 +m-1 \choose m-2} \\
&=& \sum_{q=1}^{c-2} q \Bigg[{t-n\alpha-(q+1)l+m\!\!-\!\!1 \choose m\!\!-\!\!1} -{t-n\alpha-(q+1)l+m\!\!-\!\!1-l \choose m\!\!-\!\!1} \Bigg] - {Y_0 +m-1 \choose m-2}\\
&=& \sum_{j=2}^{c-1} (j-1){X_j+m\!\!-\!\!1 \choose m\!\!-\!\!1}-\sum_{j'=3}^c (j'-2){X_{j'} +m\!\!-\!\!1 \choose m\!\!-\!\!1}- {Y_0 +m-1 \choose m-2}\\
&=& \sum_{j=2}^{c-1}(j-2){X_j+m\!\!-\!\!1 \choose m\!\!-\!\!1} + \sum_{j=2}^{c-1} {X_j+m\!\!-\!\!1 \choose m\!\!-\!\!1}-\sum_{j=3}^{c} (j-2){X_j+m\!\!-\!\!1 \choose m\!\!-\!\!1} - {Y_0 +m-1 \choose m-2}\\
&=& \sum_{j=2}^{c-1} {X_j +m\!\!-\!\!1 \choose m\!\!-\!\!1} + 0 - (c-2) {X_c+m\!\!-\!\!1 \choose m\!\!-\!\!1}- {Y_0 +m-1 \choose m-2}
\end{eqnarray*}

\vspace{0.2in}

 \noindent \textit{Partial Sums from invariants in the Pattern.}
\vspace{0.1in}

Note that $E-B = l(1-c)$.

\noindent \ding{202}
For $v=E+jc+i$ where $j=0, \dots, ln-\alpha+l-1$, $i=1, \dots, c-1$, 
\begin{eqnarray*}
\lambda_v+v &=& \lambda_0-B-jc-j-i+E+jc+i\\
&=& n\alpha+(n+1)l-1-j.
\end{eqnarray*}
Then
\begin{eqnarray*}
& & (c-1) \sum_{j=0}^{l(n+1)-\alpha-1} {t-n\alpha-l(n+1)+j+1+m\!\!-\!\!2 \choose m\!\!-\!\!2}\\
&=& (c-1) \Bigg[ {t-n\alpha-l(n+1)+m\!\!-\!\!1+l(n+1)-\alpha-1+1 \choose m\!\!-\!\!1} - {t-n\alpha-l(n+1)+m\!\!-\!\!1 \choose m\!\!-\!\!1} \Bigg]\\
&=& (c-1) \Bigg[{Y_0+m\!\!-\!\!1 \choose m\!\!-\!\!1} -{X_{n+1}+m\!\!-\!\!1 \choose m\!\!-\!\!1} \Bigg] 
\end{eqnarray*}

\noindent \ding{203}
For $v=E+jc$ where $j=1, \dots, ln-\alpha+l-1$
\begin{eqnarray*}
\lambda_v+v = \lambda_0-B-jc-j+E+jc =n\alpha+(n+1)l-j-1.
\end{eqnarray*}

Then
\begin{eqnarray*}
& & \sum_{j=1}^{l(n+1)-\alpha} {t-n\alpha-(n+1)l+j+1+m\!\!-\!\!2 \choose m\!\!-\!\!2} \\
&=& \Bigg[ {t-n\alpha-(n+1)l+m\!\!-\!\!1+l(n+1)-\alpha+1 \choose m\!\!-\!\!1} - {t-n\alpha-(n+1)l+m\!\!-\!\!1+1 \choose m\!\!-\!\!1} \Bigg]\\
&=& {Y_0+m \choose m\!\!-\!\!1} -{X_{n+1}+m \choose m\!\!-\!\!1}
\end{eqnarray*}

\vspace{0.2in}

\noindent \textit{Rewriting the Hilbert function of $I^n$.}

The Hilbert function of $I^n$ is
$$H_{I^n}(t)=  \sum_{j=1}^n \Bigg( {t-\alpha n -jl + m\!\!-\!\!1 \choose m\!\!-\!\!1} - {t-\alpha(n+1) - lj + m\!\!-\!\!1 \choose m\!\!-\!\!1}\Bigg) + {t-n\alpha+m\!\!-\!\!1 \choose m\!\!-\!\!1}$$
Now we will simplify this expression further by taking advantage of the assumption that $\alpha$ is divisible by $l$.  We can rewrite the indexing set $j=1, \dots, n$ as follows:
\begin{eqnarray*}
& & 1,2, \dots, c-1 \\
& &\hspace{0.15in} c, c+1 \dots, 2c-1\\
&& \hspace{0.15in} \vdots  \\
& & \hspace{0.15in}  (\lfloor n/c \rfloor -1)c, \dots, \lfloor n/c \rfloor \cdot c-1\\
& & \hspace{0.15in}  \lfloor n/c \rfloor \cdot c, \dots, n = \lfloor n/c \rfloor \cdot c + (n-\lfloor n/c \rfloor \cdot c)
\end{eqnarray*}

Using the fact that $c\cdot l = \alpha$, reindexing, and simplifying, the \textit{sum} in the above Hilbert Function becomes:

\begin{eqnarray*}
& & \sum_{j=1}^{c-1}  \Bigg( {t-\alpha n -jl + m\!\!-\!\!1  \choose  m\!\!-\!\!1 } - {t-\alpha(n+1) - lj +  m\!\!-\!\!1  \choose  m\!\!-\!\!1 }\Bigg) + \\
& & \sum_{p=1}^{\lfloor n/c \rfloor-1} \sum_{j=0}^{c-1}  \Bigg( {t-\alpha n -(pc+j)l + m\!\!-\!\!1  \choose  m\!\!-\!\!1 } - {t-\alpha(n+1) - l(pc+j) +  m\!\!-\!\!1  \choose  m\!\!-\!\!1 }\Bigg)+ \\
& & \sum_{j=0}^{n-\lfloor n/c \rfloor \cdot c}  \Bigg( {t-\alpha n -(\lfloor n/c\rfloor c+j)l + m\!\!-\!\!1  \choose  m\!\!-\!\!1 } - {t-\alpha(n+1) - (\lfloor n/c\rfloor c+j)l +  m\!\!-\!\!1  \choose  m\!\!-\!\!1 }\Bigg) \\
&=& \sum_{j=1}^{c-1}  \Bigg( {t-\alpha n -jl + m\!\!-\!\!1  \choose  m\!\!-\!\!1 } - {t-\alpha(n+1) - lj +  m\!\!-\!\!1  \choose  m\!\!-\!\!1 }\Bigg) + \\
& &  \sum_{j=0}^{c-1}\Bigg( \sum_{p=1}^{\lfloor n/c \rfloor-1}  {t-\alpha (n+p) -lj + m\!\!-\!\!1  \choose  m\!\!-\!\!1 } -\sum_{p'=2}^{\lfloor n/c\rfloor} {t-\alpha(n+p') - lj +  m\!\!-\!\!1  \choose  m\!\!-\!\!1 }\Bigg)+ \\
& & \sum_{j=0}^{n-\lfloor n/c \rfloor \cdot c}  \Bigg( {t-\alpha(n+\lfloor n/c \rfloor) -jl + m\!\!-\!\!1  \choose  m\!\!-\!\!1 } - {t-\alpha(n+1+ \lfloor n/c \rfloor ) - jl +  m\!\!-\!\!1  \choose  m\!\!-\!\!1 }\Bigg) \\
&=& \sum_{j=1}^{c-1}  \Bigg( {t-\alpha n -jl + m\!\!-\!\!1  \choose  m\!\!-\!\!1 } - {t-\alpha(n+1) - lj +  m\!\!-\!\!1  \choose  m\!\!-\!\!1 }\Bigg) + \\
& &  \sum_{j=0}^{c-1} {t-\alpha(n+1)-lj+ m\!\!-\!\!1  \choose  m\!\!-\!\!1 } - \sum_{j=0}^{c-1} {t-\alpha(n+\lfloor n/c \rfloor)-lj+ m\!\!-\!\!1  \choose  m\!\!-\!\!1 }\\
& & \sum_{j=0}^{n-\lfloor n/c \rfloor \cdot c}  \Bigg( {t-\alpha(n+\lfloor n/c \rfloor) -jl + m\!\!-\!\!1  \choose  m\!\!-\!\!1 } - {t-\alpha(n+1+ \lfloor n/c \rfloor ) - jl +  m\!\!-\!\!1  \choose  m\!\!-\!\!1 }\Bigg) \\
&=&\sum_{j=1}^{c-1} {t-\alpha n-lj+ m\!\!-\!\!1  \choose  m\!\!-\!\!1 } + {t-\alpha(n+1)+ m\!\!-\!\!1  \choose  m\!\!-\!\!1 } -\\
&& \sum_{j=n-\lfloor n/c \rfloor c+1}^{c-1} {t-\alpha(n+\lfloor n/c \rfloor)-lj+ m\!\!-\!\!1  \choose  m\!\!-\!\!1 }  - \sum_{j=0}^{n-\lfloor n/c\rfloor c} {t-\alpha(n+1+\lfloor n/c \rfloor )-lj+ m\!\!-\!\!1  \choose  m\!\!-\!\!1 }\\
&=& \sum_{j=1}^{c-1} {t-\alpha n -lj+ m\!\!-\!\!1  \choose  m\!\!-\!\!1 } + {t-\alpha(n+1)+ m\!\!-\!\!1  \choose  m\!\!-\!\!1 } - \sum_{p=n+1}^{c\lfloor n/c \rfloor +c-1} {t-\alpha n -lp+ m\!\!-\!\!1  \choose  m\!\!-\!\!1 } \\
& & - \sum_{p=c\lfloor n/c \rfloor}^n {t-\alpha(n+1) -lp+ m\!\!-\!\!1  \choose  m\!\!-\!\!1 }\\
\end{eqnarray*}

Letting  $X_j = t-n\alpha-lj$ and $Y_j = t-(n+1)\alpha-lj,$

\begin{eqnarray*}
&=& \sum_{j=1}^{c-1} {X_j+ m\!\!-\!\!1  \choose  m\!\!-\!\!1 } - \sum_{j'=n-c+1}^{c\lfloor n/c \rfloor -1} {t-n\alpha-l(j'+c)+ m\!\!-\!\!1  \choose  m\!\!-\!\!1 } - \sum_{j=c\lfloor n/c \rfloor}^n {Y_j+ m\!\!-\!\!1  \choose  m\!\!-\!\!1 } +\\
&& {Y_0 +  m\!\!-\!\!1 \choose  m\!\!-\!\!1 }\\
&=&\sum_{j=1}^{c-1} {X_j+ m\!\!-\!\!1  \choose  m\!\!-\!\!1 } - \sum_{j=n-c+1}^{c\lfloor n/c \rfloor -1} {t-n\alpha-\alpha-lj+ m\!\!-\!\!1  \choose  m\!\!-\!\!1 } - \\
&& \sum_{j=c\lfloor n/c \rfloor}^n {Y_j+ m\!\!-\!\!1  \choose  m\!\!-\!\!1 } + {Y_0 +  m\!\!-\!\!1 \choose  m\!\!-\!\!1 }\\
&=& \sum_{j=1}^{c-1} {X_j+ m\!\!-\!\!1 \choose  m\!\!-\!\!1 } -\sum_{j=n-c+1}^{n} {Y_j+ m\!\!-\!\!1  \choose  m\!\!-\!\!1 } + {Y_0+ m\!\!-\!\!1  \choose  m\!\!-\!\!1 }
\end{eqnarray*}
Also note that $X_c = Y_0$ so that, when $l | \alpha$, 
$$H_{I^n}(t) =  \sum_{j=1}^{c} {X_j+ m\!\!-\!\!1 \choose  m\!\!-\!\!1 } -\sum_{j=n-c+1}^{n} {Y_j+ m\!\!-\!\!1  \choose  m\!\!-\!\!1 }+{t-n\alpha+ m\!\!-\!\!1  \choose  m\!\!-\!\!1 }$$


\subsection{$\frac{3}{2} \alpha > \beta > \alpha$, $l \nmid \alpha$, $n \geq \lceil \frac{\alpha}{l} \rceil+1$}
\label{apsec:notdividecalculations}

As before, $c:=\lceil \frac{\alpha}{l} \rceil$ and $d:= \alpha \mod l$.

\vspace{0.1in}

\newpage
 
\noindent \textit{Partial sums appearing in $H_J(t)$.}

We now simplify the partial sums $\sum_v {t-\lambda_v-v+m-2 \choose m-2}$ where $v$ ranges over parts of the Build, Reverse Build, and Pattern.  We set $X_j = t-n\alpha-lj$, $Y_j=t-(n+1)\alpha-lj$, and $Z_j = t-(n+1)\alpha-lj+d$.

 \vspace{0.1in}

\noindent \textit{Partial sums from invariants in the Build and the Reverse Build.} 

These are the same as in the case $\frac{3}{2}\alpha>\beta>\alpha$, $l|\alpha$; see Section~\ref{apsec:dividescalculations} for details.
 
  \vspace{0.1in}

\noindent \textit{Partial sums from invariants in the Pattern}

Recall that $E-B = l(1-c)$.  

\noindent \ding{202} For $v = E+p\alpha+jc+i$ where $p=0, \dots, n-c$, $j=0,\dots d-1$, $i=1, \dots, c-1$,
\begin{eqnarray*} 
\lambda_v+v &=& \lambda_0-B-pl-p\alpha-jc-j-i+E+p\alpha+jc+i\\
&=& (n+1)\alpha+nl-1+l(1-c)-pl-j\\
&=&(n+1)\alpha+l(n-c-p+1)-j-1.
\end{eqnarray*}

Then

\begin{eqnarray*}
& & (c-1)\Bigg[\sum_{p=0}^{n-c} \sum_{j=0}^{d-1} {t-(n+1)\alpha-l(n-c-p+1)+j+1+m\!\!-\!\!2 \choose m\!\!-\!\!2}\Bigg]\\
&=& (c-1) \Bigg[\sum_{p=0}^{n-c} \Bigg( {t-(n+1)\alpha-l(n-c-p+1)+d +m\!\!-\!\!1+d\choose m\!\!-\!\!1} \\
& &- {t-(n+1)\alpha-l(n-c-p+1)+m\!\!-\!\!1 \choose m\!\!-\!\!1} \Bigg) \Bigg]\\
&=& (c-1)\Bigg[\sum_{j=1}^{n-c+1} \Bigg( {Z_j + m\!\!-\!\!1  \choose  m\!\!-\!\!1 } - {Y_j+ m\!\!-\!\!1  \choose  m\!\!-\!\!1 } \Bigg) \Bigg].
\end{eqnarray*}

\noindent \ding{203} For $v = E+p\alpha+jc$ where $p=0, \dots, n-c$, $j=1, \dots, d$, 
\begin{eqnarray*}
\lambda_v+v &=&\lambda_0-B-pl-p\alpha-jc-j+E+p\alpha+jc\\
&=&(n+1)\alpha+nl-1+l(1-c)-pl-j\\
&=&(n+1)\alpha+l(n-c-p+1)-j-1.
\end{eqnarray*}

Then

\begin{eqnarray*}
& & \sum_{p=0}^{n-c} \sum_{j=1}^{d} {t-(n+1)\alpha-l(n-c-p+1)+j+1+m\!\!-\!\!2 \choose m\!\!-\!\!2}\\
&=& \sum_{p=0}^{n-c} \Bigg[ {t-(n+1)\alpha-l(n-c-p+1)+d+1+m\!\!-\!\!1\choose m\!\!-\!\!1} \\
& & -{t-(n+1)\alpha-l(n-c-p+1)+1+m\!\!-\!\!1 \choose m\!\!-\!\!1} \Bigg]\\
&=& \sum_{j=1}^{n-c+1} \Bigg[  {Z_j+m \choose m\!\!-\!\!1} - {Y_j+m \choose m\!\!-\!\!1}\Bigg].
\end{eqnarray*}

\noindent\ding{204} For $v=E+p\alpha+dc+j(c-1)+i$ where $p=0, \dots, n-c-1$, $j=0, \dots, l-d-1$, $i=1, \dots, c-2$,
\begin{eqnarray*}
\lambda_v+v &=& \lambda_0-B-pl-p\alpha-dc-d-jc-i+E+p\alpha+dc+jc-j+i\\
&=&(n+1)\alpha+nl-1+l(1-c)-pl-d-j\\
&=&(n+1)\alpha+l(n-c-p+1)-j-d-1.
\end{eqnarray*}

Then

\begin{eqnarray*}
& & (c-2) \Bigg[  \sum_{p=0}^{n-c-1} \sum_{j=0}^{l-d-1} {t-(n+1)\alpha-l(n-c-p+1)+j+d+1+m\!\!-\!\!2 \choose m\!\!-\!\!2}\Bigg]\\
&=& (c-2) \Bigg[  \sum_{p=0}^{n-c-1} \Bigg(  {t-(n+1)\alpha-l(n-c-p+1)+d+l-d+m\!\!-\!\!1\choose m\!\!-\!\!1} \\
&&-{t-(n+1)\alpha-l(n-c-p+1)+d+m\!\!-\!\!1 \choose m\!\!-\!\!1}\Bigg)\Bigg]\\
&=&  (c-2)\sum_{j=1}^{n-c} {Y_j+m\!\!-\!\!1 \choose m\!\!-\!\!1} - (c-2) \sum_{j'=2}^{n-c+1} {Z_{j'}+m\!\!-\!\!1 \choose m\!\!-\!\!1}.
\end{eqnarray*}

\noindent \ding{205}  For $v=E+p\alpha+dc+j(c-1)$ where $p=0, \dots, n-c-1$, $j=1, \dots, l-d$,
\begin{eqnarray*}
\lambda_v+v &=& \lambda_0-B-pl-p\alpha-dc-d-jc+E+p\alpha+dc+jc-j\\
&=&(n+1)\alpha+nl-1+l(1-c)-pl-d-j\\
&=& (n+1)\alpha+l(n-c-p+1)-d-j-1.
\end{eqnarray*}

Then

\begin{eqnarray*}
& &  \sum_{p=0}^{n-c-1} \sum_{j=1}^{l-d} {t-(n+1)\alpha-l(n-c-p+1)+j+d+1+m\!\!-\!\!2 \choose m\!\!-\!\!2}\\
&=& \sum_{p=0}^{n-c-1} \Bigg[  {t-(n+1)\alpha-l(n-c-p)+d+l-d+1+m\!\!-\!\!1 \choose m\!\!-\!\!1} \\
&&-{t-(n+1)\alpha-l(n-c-p+1)+d+m \choose m\!\!-\!\!1}\Bigg]\\
&=& \sum_{j=1}^{n-c} {Y_j+m \choose m\!\!-\!\!1} -\sum_{j'=2}^{n-c+1} {Z_{j'}+m \choose m\!\!-\!\!1}
\end{eqnarray*}

\vspace{0.2in}

\noindent \textit{Rewriting the Hilbert function.}

The expression for the Hilbert function of $I^n$ in terms of $\alpha$ and $l$ found in Section \ref{sec:hilbmethod} is 
$$H_{I^n}(t)=  \sum_{j=1}^n \Bigg( {t-\alpha n -jl + m\!\!-\!\!1  \choose  m\!\!-\!\!1 } - {t-\alpha(n+1) - lj +  m\!\!-\!\!1  \choose  m\!\!-\!\!1 }\Bigg) + {t-n\alpha+ m\!\!-\!\!1  \choose  m\!\!-\!\!1 }.$$

Using the fact that $lc = \alpha+l-d$ the summation in this expression is
\begin{eqnarray*}
& & \sum_{j=1}^{c-1} {t-n\alpha-lj+ m\!\!-\!\!1  \choose  m\!\!-\!\!1 } + \sum_{j=c}^n {t-n\alpha-lj+ m\!\!-\!\!1  \choose  m\!\!-\!\!1 } - \sum_{j=1}^n {t-(n+1)\alpha-lj+ m\!\!-\!\!1  \choose  m\!\!-\!\!1 }\\
&=& \sum_{j=1}^{c-1} {t-n\alpha-lj+ m\!\!-\!\!1  \choose  m\!\!-\!\!1 } +\sum_{i=0}^{n-c} {t-n\alpha-l(c+i) + m\!\!-\!\!1 \choose  m\!\!-\!\!1 }- \sum_{j=1}^n {t-(n+1)\alpha-lj+ m\!\!-\!\!1  \choose  m\!\!-\!\!1 }\\
&=& \sum_{j=1}^{c-1} {t-n\alpha-lj+ m\!\!-\!\!1  \choose  m\!\!-\!\!1 } +\sum_{i=0}^{n-c} {t-n\alpha-\alpha-l+d-li + m\!\!-\!\!1 \choose  m\!\!-\!\!1 }\\
&&- \sum_{j=1}^n {t-(n+1)\alpha-lj+ m\!\!-\!\!1  \choose  m\!\!-\!\!1 }\\
 &=& \sum_{j=1}^{c-1} {t-n\alpha-lj+ m\!\!-\!\!1  \choose  m\!\!-\!\!1 } +\sum_{j'=1}^{n-c+1} {t-(n+1)\alpha-lj'+d + m\!\!-\!\!1 \choose  m\!\!-\!\!1 }\\
 &&- \sum_{j=1}^n {t-(n+1)\alpha-lj+ m\!\!-\!\!1  \choose  m\!\!-\!\!1 }\\
\end{eqnarray*}
 This leads to the formula for $H_{I^n}(t)$ found in Section~\ref{sec:notdividerewrite}.
 
 \vspace{0.2in}


\subsection{$\frac{3}{2} \alpha > \beta > \alpha$, $2 \leq n < \frac{\alpha}{l} +1$}
\label{apsec:closesmallncalculations}

\noindent \textit{Partial sums appearing in $H_J(t)$.}

We now simplify the partial sums $\sum_v {t-\lambda_v-v+m-2 \choose m-2}$ where $v$ ranges over parts of the Build, Reverse Build, and Pattern.  As above, we set $X_j = t-n\alpha-lj$, $Y_j=t-(n+1)\alpha-lj$, and $Z_j = t-(n+1)\alpha-lj+d$.

 \vspace{0.1in}
 \noindent \textit{Partial sums from invariants in the Build.}

\noindent \ding{202} For $v=0, \dots, l$,
$$\lambda_0-2v+v = \lambda_0-v.$$

Then

\begin{eqnarray*}
& & \sum_{j=0}^l {t-(n+1)\alpha-nl+1+j+m-2 \choose m-2}\\
&=& {t-(n+1)\alpha-(n-1)l+m-1+1 \choose m-1} -{t-(n+1)\alpha-nl+m-1 \choose m-1}\\
&=& {Y_{n-1} + m \choose m-1} - {Y_n+m-1 \choose m-1}
\end{eqnarray*}

\noindent \ding{203} For $v = l(1+ \cdots +q)+(q+1)j$ where $q=1, \dots, n-2$ and $j=1, \dots, l$, 
\begin{eqnarray*}
\lambda_v+v &=& \lambda_0 - (2+\cdots +(q+1))l-(q+1)j-j+l(1+\cdots q)+(q+1)j\\
&=& \lambda_0-ql-j.
\end{eqnarray*}

Then

\begin{eqnarray*}
& & \sum_{q=1}^{n-2} \sum_{j=1}^l {t-(n+1)\alpha-nl+1+ql+j+m-2 \choose m-2}\\
&=& \sum_{q=1}^{n-2} \Bigg[ {t-(n+1)\alpha-(n-q)l+m-1+l+1 \choose m-1} - {t-(n+1)\alpha-(n-q)l+m-1+1 \choose m-1}\Bigg]\\
&=& \sum_{j=1}^{n-2} {Y_j+m \choose m-1} - \sum_{j'=2}^{n-1} {Y_{j'} +m \choose m-1}\\
&=& {Y_1+m \choose m-1} - {Y_{n-1}+m \choose m-1}.
\end{eqnarray*}

\noindent \ding{204} For $v = l(1+\cdots +q)+(q+1)j-x$ where $q=1, \dots, n-2$, $j=1, \dots, l$, and $x=1, \dots, q$,

\begin{eqnarray*}
\lambda_v+v &=& \lambda_0 - (2+ \cdots+(q+1))l-(q+1)j+x-j+1+l(1+\cdots+q) +(q+1)j-x\\
&=&\lambda_0-(q+1)l_l-j+1.
\end{eqnarray*}

Then

\begin{eqnarray*}
& & \sum_{q=1}^{n-2} q \Bigg[ \sum_{j=1}^l {t-(n+1)\alpha-nl+ql+j+m-2 \choose m-2} \Bigg]\\
&=& \sum_{q=1}^{n-2} q \Bigg[ {t-(n+1)\alpha-l(n-q-1)+m-1 \choose m-1} -{t-(n+1)\alpha-l(n-q)+m-1\choose m-1} \Bigg]\\
&=& \sum_{j=1}^{n-2}(n-j-1){Y_j+m-1 \choose m-1} - \sum_{j'=2}^n (n-j') {Y_{j'}+m-1 \choose m-1}\\
&=& \sum_{j=1}^{n-2} (n-j-1) {Y_j+m-1 \choose m-1} - \sum_{j=2}^{n-1}(n-j-1) {Y_j+m-1 \choose m-1} - \sum_{j=2}^{n-1} {Y_j+m-1 \choose m-1}\\
&=&(n-2){Y_1+m-1 \choose m-1} - \sum_{j=2}^{n-1} {Y_j+m-1 \choose m-1}.
\end{eqnarray*}

 \vspace{0.1in}
 \noindent \textit{Partial sums from invariants in the Reverse Build.}

\noindent \ding{202} For $v = (k-1)-j$ where $j=0, \dots, l$,
$$\lambda_v+v = (k-1)-j+\lambda_{k-1}+2j = l+k+j.$$

Then

\begin{eqnarray*}
\sum_{j=0}^l {t-n\alpha -l-j+m-2 \choose m-2} &=& \Bigg[ {t-n\alpha-l+m-2+1 \choose m-1} - {t-n\alpha-l+m-2-l \choose m-1}\Bigg]\\
&=& {X_1+m-1 \choose m-1} - {X_2 +m-2 \choose m-1}.
\end{eqnarray*}

\noindent \ding{203} For $v = (k-1)-(l(1+\cdots +q)+(q+1)j)$ where $q=1, \dots, n-2$ and $j=0, \dots, l$

\begin{eqnarray*}
\lambda_v+v &=& (k-1)-l(1+\cdots+q)-j(q+1)+\lambda_{k-1}+(2+ \dots+(q+1))l+(q+1)j+j\\
&=&n\alpha-1-l+l+1+(q+1)l+j\\
&=&n\alpha+(q+1)l+j.
\end{eqnarray*}
Note that $\lambda_v$ such that $q=n-2$ and $j=l$ is not in the Reverse Build.

\begin{eqnarray*}
& & \sum_{q=1}^{n-2} \sum_{j=1}^l {t-n\alpha-(q+1)l-j+m-2 \choose m-2} - {t-n\alpha-(n-1)l-l+m-2 \choose m-2}\\
&=&\sum_{q=1}^{n-2} \Bigg[  {t-n\alpha-(q+1)l+m-2 \choose m-1} - {t-n\alpha-(q+1)l +m-2 -l \choose m-1} \Bigg] - {X_n+m -2\choose m-2}\\
&=& \sum_{j=2}^{n-1} {X_j+m-2 \choose m-1} - \sum_{j'=3}^n {X_{j'}+m-2\choose m-1}- {X_n+m -2\choose m-2}\\
&=& {X_2+m-2 \choose m-2} - {X_n+m-2 \choose m-1}- {X_n+m -2\choose m-2}
\end{eqnarray*}

\noindent \ding{204} For $v = (k-1)-(l(1+\cdots+q)+(q+1)j-x)$ where $q=1, \dots, n-2$, $j=1, \dots, l$, and $x=1, \dots, q$, 

\begin{eqnarray*}
\lambda_v +v &=& n\alpha-l(1+\cdots q)-(q+1)j+x+\lambda_{k-1}+(2+\cdots+(q+1))l+(q+1)j-x+j-1\\
&=&n\alpha+ql+l+1+j-1 \\
&=&n\alpha+l(q+1)+j-1.
\end{eqnarray*}
Note that $\lambda_v$ such that $q=n-2$, $j=l$, and $x=1$ is not in the Reverse Build.

\begin{eqnarray*}
& & \sum_{q=1}^{n-2} q \Bigg[ \sum_{j=1}^l {t-n\alpha-l(q+1)-j+1+m-2 \choose m-2} \Bigg] - {t-n\alpha-l(n-1) - l+1+m-2 \choose m-2}\\
&=& \sum_{q=1}^{n-2}q \Bigg[ {t-n\alpha-l(q+1)+m-2 \choose m-1} - {t-n\alpha-l(q+1) +m-1 -l \choose m-1} \Bigg] - {X_n+m-1 \choose m-2}\\
&=& \sum_{j=2}^{n-1}(j-1){X_j+m-1 \choose m-1} - \sum_{j'=3}^n (j'-2) {X_{j'}+m-1 \choose m-1} - {X_n+m-1 \choose m-2}\\
&=& \sum_{j=2}^{n-1} {X_j+m-1 \choose m-1} + \sum_{j=2}^{n-1} (j-2) {X_j+m-1\choose m-1} - \sum_{j=3}^n (j-2){X_j+m-1 \choose m-1} - {X_n+m-1 \choose m-2}\\
&=& \sum_{j=2}^{n-1} {X_j+m-1 \choose m-1} - (n-2){X_n+m-1 \choose m-1}- {X_n+m-1 \choose m-2}
\end{eqnarray*}

 \vspace{0.1in}
 \noindent \textit{Partial sums from invariants in the Pattern.}

Note that $E-B = l-ln$.

\noindent \ding{202} For $v=E=jn+i$ where $j=0, \dots, \beta-nl-1$ and $i=1, \dots, n-1$,
\begin{eqnarray*}
\lambda_v+v &=& \lambda_0-B-jn-j-i+E+jn+i\\
&=& \lambda_0+l-ln-j.
\end{eqnarray*}

Then

\begin{eqnarray*}
&&(n-1) \sum_{j=0}^{\beta-nl-1} {t-(n+1)\alpha-nl+1 -l +ln+j+m-2 \choose m-2} \\
&=& (n-1) \sum_{j=0}^{\beta-nl-1} {t-(n+1)\alpha-l+j+m-1 \choose m-2}\\
&=&(n-1) \Bigg[ {t-(n+1)\alpha-l+m-1 +\beta-nl \choose m-1} - {t-(n+1)\alpha-l+m-1 \choose m-1}\Bigg]\\
&=& (n-1) \Bigg[ {t-n\alpha-nl+m-1 \choose m-1} - {Y_1+m-1\choose m-1}\Bigg]\\
&=& (n-1) \Bigg[ {X_n+m-1 \choose m-1} - {Y_1+m-1 \choose m-1} \Bigg].
\end{eqnarray*}

\noindent \ding{203} For $v=E+jn$ where $j=1, \dots, \beta-nl-1$, 
\begin{eqnarray*}
\lambda_v+v &=&\lambda_0-B-jn-j+e+jn\\
&=& (n+1)\alpha+nl-1+l-ln-j\\
&=&(n+1)\alpha+l-1.
\end{eqnarray*}

Then

\begin{eqnarray*}
& & \sum_{j=1}^{\beta-nl-1} {t-(n+1)\alpha-l+1+j+m-2 \choose m-2}\\
&=& \Bigg[ {t-(n+1)\alpha-l+m-1+\beta-nl\choose m-1} - {t-(n+1)\alpha-l+m-1+1 \choose m-1}\Bigg]\\
&=& {X_n+m-1 \choose m-1}- {Y_1+m \choose m-1}.
\end{eqnarray*}

\bibliography{Ginof2CI}
\bibliographystyle{amsalpha}
\nocite{*}

\end{document}